\documentclass[11 pt]{article}

\usepackage{amssymb,amsmath, amsthm,amstext,amscd,subfigure,dsfont}
\usepackage{breqn, color,enumerate}
\usepackage{hyperref}
\usepackage[english]{babel}
\usepackage{fullpage}
\usepackage[affil-it]{authblk}

\makeatletter
\newsavebox{\@brx}
\newcommand{\llangle}[1][]{\savebox{\@brx}{\(\m@th{#1\langle}\)}%
  \mathopen{\copy\@brx\mkern2mu\kern-0.9\wd\@brx\usebox{\@brx}}}
\newcommand{\rrangle}[1][]{\savebox{\@brx}{\(\m@th{#1\rangle}\)}%
  \mathclose{\copy\@brx\mkern2mu\kern-0.9\wd\@brx\usebox{\@brx}}}
\makeatother

\numberwithin{equation}{section}

\newtheorem{thm}{Theorem}

\newtheorem{cor}{Corollary} 
\newtheorem{lemma}{Lemma}
 
\newtheorem{remark}{Remark}
\newtheorem{dfn}{Definition}
\newtheorem{asm}{Assumption}

\newcommand{\Pbb}{\mathbb{P}}
\newcommand{\Ebb}{\mathbb{E}}
\newcommand{\CalF}{\mathcal{F}}
\newcommand{\Nbb}{\mathbb{N}}
\newcommand{\Rbb}{\mathbb{R}}
\newcommand{\Dbb}{\mathbb{D}}
\newcommand{\Cbb}{\mathbb{C}}
\newcommand{\Xbb}{\mathbb{X}}
\newcommand{\calP}{\mathcal{P}}
\newcommand{\CalB}{\mathcal{B}}
\newcommand{\toP}{\stackrel{p}{\rightarrow}}
\DeclareMathOperator*{\essinf}{ess\,inf}
\DeclareMathOperator*{\esssup}{ess\,sup}

\title{Many-Server Queues with Random Service Rates in the Halfin-Whitt Regime: A Measure-Valued Process Approach}
\author[1]{Burak B\"{u}ke}
\author[1]{Wenyi Qin} 
\affil[1]{School of Mathematics, University of Edinburgh, Edinburgh, Scotland, UK \authorcr {\tt B.Buke@ed.ac.uk, W.Qin-3@sms.ed.ac.uk}}
\date{\today}

\begin{document}

\maketitle
\begin{abstract}
We consider many-server queueing systems with heterogeneous exponential servers and renewal arrivals. The service rate of each server is a random variable drawn from a given distribution. We develop a framework for analyzing the heavy traffic limit of these queues in random environment using probability measure-valued stochastic processes. We introduce the measure-valued fairness process which denotes the proportion of cumulative idleness experienced by servers whose rates fall in a Borel subset of the support of the service rates. It can be shown that these fairness processes do not converge in the usual Skorokhod-$J_1$ topology, hence we introduce a new notion of convergence based on shifted versions of these processes. We also introduce some useful martingales to identify limiting fairness processes under different routing policies. 
\end{abstract}

\section{Introduction}

Many server queues have been the subject of much research due to their applicability in large scale service systems, especially in call centers. Exact analysis of many-server systems is generally intractable and one resorts to approximation methods relying on functional strong law of large numbers and functional central limit theorems. In this work, we adopt the scaling introduced in the seminal paper of Halfin and Whitt~\cite{halfinwhitt81}, where they show that it is possible to achieve high quality of service along with high utilization of resources. This is achieved by setting the number of servers to  what is required to stabilize the system, generally referred as offered load plus an amount proportional to the square root of the offered load. 

The conventional analysis in the literature following~\cite{halfinwhitt81} focuses on either identical servers or servers classified in finitely many pools where servers are identical within each pool. However, in many real world applications the servers are humans who have inherently different abilities and serve with different rates which depends on the individual abilities, mood and health of the person. Hence, in the modeling process, each server requires individual attention and this generally results in either loss of Markov property and/or explosion of the state space dimensions. 

A series of papers by Atar and his colleagues~\cite{atar08,atarschwartzshaki11, atarschwartz08} tackle this individuality problem by assuming that service rates of  servers in the systems are independent and identically distributed (henceforth referred as i.i.d.) random variables that are realized at the beginning and constant through time but may not be available to the system controller. Atar~\cite{atar08} investigates two routing policies, namely longest-idle-server first and fastest server first, in an \emph{ad hoc} manner. Under these two policies Atar~\cite{atar08} show that the many-server systems with random service rates can be approximated by one-dimensional diffusion processes with a random drift coefficient, which has the same structure for both policies,  and a mean reversion coefficient which depends on the routing policy. Our setting in this paper is similar to the i.i.d.\ service rate setting of the work of Atar and his colleagues. The main difference is that we develop a general framework relying on probability measure-valued processes without assuming a particular routing policy and propose a generic representation for the parameters of the limiting diffusion. Technically speaking, Atar~\cite{atar08} adopts a  Riemann-type approach in the analysis by dividing the support of the random service rates into small intervals and  shows convergence as these intervals become finer. In contrast, ours is more of a Lebesgue-type analysis where we consider Borel subsets of the support of service rates and introduce a measure relying on how the total idleness is distributed among servers with different service rates. This allows us to treat the  general case without assuming a specific routing policy.

Another important feature omitted in the aforementined literature with i.i.d.\ service rates is the consideration of customer abandonments. Garnett et al.~\cite{Garnett2002} show that abandonments plays an important role in the design of many-server systems. We believe that abandonment behavior is particularly important when service rate uncertainty is present. In the Halfin-Whitt scaling as introduced in~\cite{atar08}, there is always a positive probability that the total service capacity is less than the offered load and hence, the queueing system is unstable. This is also reflected in the diffusion limits derived in~\cite{atar08} as the probability of the drift coefficient being positive, which makes a thorough steady-state analysis of the system impossible. 

The key concept in  our analysis is the probability measure-valued fairness process which is a left-continuous process indicating the proportion of total cumulative idleness experienced by servers belonging to a certain set. To  have a well-defined fairness process,  this process assumes an arbitrary probability measure as a constant value until some of the servers in the system experiences idleness and is equal to the self-normalized cumulative idleness process after the total cumulative idleness is positive. The point of singularity where the system experiences idleness for the first time raises difficulties in the analysis. This singularity point is the only possible discontinuity of the measure-valued processes and it can be shown that the fairness processes do not converge in any of the four topologies introduced by Skorokhod~\cite{sko56}. To overcome this issue, we define a new notion of convergence based on the shifted versions of the fairness process around this singularity point. We then relate this new notion of convergence to the convergence of Radon measure-valued cumulative idleness processes (before the normalization) in the usual Skorokhod-$J_1$ topology. We demonstrate the use of this result by deriving limiting fairness measures for priority-based policies such as fastest-server-first and slowest-server-first. We then show that the policy dependent parameter of the diffusion limits derived in~\cite{atar08} is the expected value of a random variable following the probability measure given by the fairness process. 

As the next step, we focus on the problem of identifying the fairness process for specific policies. We introduce a sequence of martingales related to the total cumulative idleness experienced by the servers belonging to a given set. We show that under some mild conditions these martingales converge in probability to 0 and this provides us an alternative characterization of the weak limit of the cumulative idleness process. Then, we use this alternative characterization to show that the policy dependent parameter identified by Atar~\cite{atar08} for the longest-idle-server-first policy is the same for a much more general class of policies, which we call totally blind policies. We also provide explanation for the structure of this parameter relating it to the basic properties of minimum of independent exponential random variables. 

\subsection{Notation}

We denote the set of real numbers, set of positive real numbers and the set of positive integers as $\Rbb, \Rbb_{+}$ and $\Nbb$, respectively. To simplify notation, for any $a,b\in\Rbb$, we define 
\[a\wedge b:=\min\{a,b\} \mbox{ and }a\vee b:=\max\{a,b\}.\]
For any separable metric space $\Xbb$, $\CalB(\Xbb)$ denotes the set of all Borel subsets of $\Xbb$. We assume that all the random elements defined in this paper lies in the generic probability space $(\Omega, \CalF, \Pbb)$,  unless stated otherwise. We also define $\calP$ to denote the space of probability measures defined on $(\Rbb_{+}, \CalB(\Rbb_+))$ equiped with Prokhorov metric~\cite{bil99}. We denote $\Pbb^n\Rightarrow \Pbb$ if probability measures converge with respect to Prokhorov metric, which is equivalent to weak convergence of these probability measures. With a slight abuse of notation, we denote $X^n\Rightarrow X$, if the probability laws of the random elements $X^n$ weakly converge to the probability law of $X$. We denote the $\sigma$-field generated by any random element $X$ as $\sigma(X)$.

We denote the spaces of continuous, bounded continuous, right-continuous and left-continuous functions that map interval $[a,b]\subset \Rbb$ to $\Xbb$ as $C_\Xbb[a,b], C_\Xbb^b[a,b],  D_\Xbb[a,b]$ and $G_\Xbb[a,b]$, respectively. The space $C_\Xbb[a,b]$ is equipped with the topology of uniform convergence, i.e., the supremum norm, and the spaces $D_\Xbb[a,b]$ is equipped with Skorokhod-$J_1$ topology and $d_S$ denote the usual Skorokhod-$J_1$metric. If $b<\infty$ and $f(t)\in G_\Xbb[a,b]$, then $g_f(t):=f(b-t)\in D_\Xbb[a-b, 0]$ and we define the Skorokhod-$J_1$ metric on $G_\Xbb[a,b]$ as $d_S'(f_1,f_2):=d_S(g_{f_1}, g_{f_2})$ for any $f_1,f_2\in G_\Xbb[a,b]$. Following Whitt~\cite{Whitt1970}, for $f_n,f\in D_\Xbb[a,\infty)$ $f_n\to f$ in Skorokhod-$J_1$ topology if and only if for any $T>a$, the modifications of these functions on $D_{\Xbb}[a,T]$ converge in Skorokhod-$J_1$ topology. The topology on $G_\Xbb[a,\infty)$ is defined in a similar manner. For any function $f:[0,T]\to \Rbb$, we also define 
\[
|f|_{*,T}:=\sup_{0\leq t\leq T}|f(t)|, f^+(t):=\max\{f(t),0\} \mbox{ and }  f^-(t):=\max\{-f(t),0\}.
\]
For functions $f:\Xbb\to \Rbb$ and measure $\zeta$ defined on $(\Xbb, \CalB(\Xbb))$, we use the inner-product notation for the integral
\[
\llangle[\big]f,\zeta\rrangle[\big]:=\int_{\Xbb}f(x)d\zeta(x).
\] 
Defining the identity function $\iota:\Rbb\to \Rbb$ such that $\iota(x)=x$ for all $x\in\Rbb$, the  expected value of a random variable following the distribution $\zeta\in \calP$ can be denoted as $\llangle[\big]\iota,\zeta\rrangle[\big]$.

\section{Related Literature}
This work is a part of the literature initiated by the seminal work of Halfin and Whitt~\cite{halfinwhitt81}, where they develop diffusion approximations by scaling arrival rates along with the number of exponential servers. This work has important practical impications as it shows both quality of service and efficiency of the system can be achieved by using the so-called square-root safety-staffing, i.e., setting the staffing level to be the number required to stabilize the incoming load plus a multiple of the square root of this number. The main line of research in this direction assumes that all the servers are identical, i.e., the service times of customers follow an exponential distribution with a known rate independent of the chosen server. However, in most cases the servers are humans with inherently different abilities and serve customers at different rates. Gans et al.~\cite{ganlmsy10} provide a thorough numerical analysis of how the server heterogeneity affects the performance based on real call center data. 

A common approach in modeling server heterogeneity is to assume that the servers can be grouped as ``server pools'' and that service rate can vary between pools, but all servers are identical within the same pool and serve with the same rate (see, e.g. \cite{armony05}). We believe that there are several limitations of this approach. First, this approach does not explicity model the inherent individual differences between servers due to human nature. Second, in general the rate of service in each pool are assumed to be known, or at the very least the pool that each server belongs to is known, which is not always possible in practice. We believe that these limitations can be remedied by modeling the service rates as random variables following the same distribution. 

In the last decade, there has been some research effort to analyze queueing systems with random parameters. Following the methodology introduced by Harrison and Zeevi~\cite{harzee05}, fluid limits are used in conjuction with stochastic programming to characterize how queueing systems should be designed and controlled under parameter uncertainty. Much of this research concentrates on arrival rate uncertainty due to forecasting errors~\cite{basharzee:10,zanhasmor18}. Indeed, Bassamboo et al.~\cite{basharzee:10} show that if the coefficient of variation of random arrival rate is greater than a certain threshold, fluid limits for the many-server systems yield more reliable approximations compared with the diffusion limits. The literature on the uncertainty related to service is relatively few. In a recent paper, Ibrahim~\cite{ibrahim19} studies the problem of staffing many-server queueing systems when the actual number of servers is random using fluid limits.

Our current work is most closely related to the excellent paper by Atar~\cite{atar08}. In~\cite{atar08}, Atar develops diffusion limits for many-server systems with random service rates for two routing policies, namely longest-idle-server-first and fastest-server-first in an \emph{ad hoc} manner. To achieve this, he partitions the support of the random service rates into small intervals and approximates the system as an inverted-V system (as studied in Armony~\cite{armony05}) where each interval corresponds to a pool of servers. Under the stated routing policies, the inverted-V systems admit a one-dimensional diffusion limit for any finite number of pools, exhibiting a phenomenon  referred as state-space collapse in the queueing literature and as the  intervals become finer the inverted-V system and the system with random service rates converge together to the same limit. In a related paper, Atar and Shwartz~\cite{atarschwartz08} show that it is possible to efficiently operate a system with random service rates by sampling only a small portion of the service rates. Atar, Shaki and Shwartz~\cite{atarschwartzshaki11} suggest a blind policy, i.e., a policy which does not require the exact knowledge of service rates, to equalize the cumulative idleness experienced by servers. 

In this paper, we analyze the system with i.i.d.\ random service rates introduced in~\cite{atar08} without assuming any particular routing policy. To achieve this generality, we use stochastic processes assuming values in the space of probability measures. Measure-valued processes have been used to obtain fluid limits for single server queueing systems (e.g., \cite{decmoyal08,gropuhawil02,grorobzwart08}) and more recently for many-server queues by Kaspi and Ramanan~\cite{kasram11}, Kang and Ramanan~\cite{kangram10} and Reed and Talreja~\cite{reedtal15}. In a more recent paper, Kaspi and Ramanan~\cite{kasram13} also use measure-valued processes to come up with limits involving stochastic partial differential equations to many-server queueing systems. In this line of work, Ramanan and co-authors use the  number of customers who has a certain age in the service or queue to define the measure-valued processes. Different from the aforementioned work, we use a measure-valued process to keep track of the proportion of cumulative idleness experience by different servers in the system. This definition is motivated by the idleness-ratio used to design routing policies to control finite pool systems~\cite{gurwhitt09,wararm13}. 

\section{Dynamics of the System Processes}

In this work, we consider an infinite sequence of queueing systems, where the $n$th system has $n$ servers. The arrivals at the $n$th system occur according to a renewal process with rate $\lambda^n$. More specificallly, assuming that  $u_k^n$s are i.i.d.\ random variables with mean $1$ for each $k, n\in\Nbb$, the number of arrivals at the $n$th system by time $t\geq 0$ is
\[
A^n(t)=\sup\left\{k: \sum_{i=1}^k\frac{u_i^n}{\lambda^n}\leq t\right\}.
\]
Each customer can be served by any server in the system. However, the service time of each customer depends on the agent serving her/him and if the customer is served by agent $k$ $(1\leq k \leq n)$, the service time is distributed exponentially with rate $\tilde{\mu}^n_k$. The service rates $\tilde{\mu}_k^n$s are positive i.i.d.\ random variables with a general distribution $F(\mu)$ and  are constant through time for each server. We assume that $\bar{\mu}=\Ebb[\tilde{\mu}_k^n]$ and $\Ebb[(\tilde{\mu}_k^n)^2]<\infty$. We also assume the following Halfin-Whitt type heavy traffic condition between arrival and the expected service rates:
\begin{asm}
\label{asm:HW}
As $n\to \infty$, $n^{-1/2}(\lambda^n-n\bar{\mu})\to \hat{\lambda}$, where $\hat{\lambda}\in \Rbb$.
\end{asm}
Each customer has an i.i.d.\ exponential patience time with rate $\gamma$ and abandons the system if her/his service has not commenced until the patience time expires. Once the service starts, the customer stays in the system until the service is finished. The customers are served on a first-come-first-served basis and if more than one server is idle when a customer arrives, the customer is routed to a server according to a pre-specified service discipline. At this point, the only restrictions we impose on the service discipline are (i) it is non-idling, i.e., servers cannot stay idle if there is work in the system, (ii) it is non-anticipating, i.e., when the routing is done the server is chosen based on the information available so far and (iii) it is non-preemptive, i.e., once a service starts for a customer at a server, it continues at the same server until the service finishes.

For the $n$th system, the number of customers in the system at time $t$, the number of customers routed to the $k$th agent and the number of customers who completes service at agent $k$ by time $t$ are denoted $X^n(t), R_k^n(t)$ and $D_k^n(t)$, respectively. The idleness process for each agent, $I^n_k(t)$,  is defined to be 1 if the $k$th server in the $n$th system is idle at time $t$ and 0 otherwise. The idleness processes play a critical role in our analysis and can be related to the routing  and departure processes as
\[
I_k^n(t)=I_k^n(0)-R_k^n(t)+D_k^n(t), \text{ for all }t\geq 0, 1\leq k \leq n.
\]
 The routing process is a counting process and the customers are routed to server $k$ only when it is idle, which implies 
\[
(1-I_k^n(t-))dR_k^n(t)=0, \mbox{ for all $t\geq 0$ and $1\leq k\leq n$}. 
\]
The non-idling property of the policies can be expressed as
 \begin{equation}
  (X^n(t)-n)^-=\sum_{k=1}^n I_k^n(t), \forall t\geq 0, n\in \Nbb.
  \label{Eq:nonidling_basic}
 \end{equation}
Due to the exponential nature of service and patience times, the counting processes for service completions and abandonments can be modelled as time changes and/or thinings of homogeneous Poisson processes. Pang et al.~\cite{ptw07} provide a detailed overview of pathwise construction of  such processes. We take $\{S^n(t), t\geq 0\}$ to be Poisson processes with rate 1 and event epochs $\{\theta_i^n\}$. We also take a sequence of independent uniform(0,1) random variables, $U_i^n$, and define the sequence of random variables
\[
\kappa_i^n:= \min\left\{ k: \left(\sum_{j=1}^{n}\tilde{\mu}_j^n\right)U_i^n \leq \sum_{j=1}^{k}\tilde{\mu}_j^n\right\}, \mbox{ for all }i,n\in \mathbb{N}.
\]
To construct the departure process for each server, we first consider $\{S_P^n(t)=S(\sum_{i=1}^n\tilde{\mu}_k^nt), t\geq 0\}$, a time change of the standard Poisson process, to act as the potential service completion process for the $n$th system, i.e., the event epochs of this process are the potential candidates for the actual service completion times. Using splitting property of Poisson processes, if $\kappa_i^n=k$, we treat the $i$th event of the potential service completion process $S_P^n$ occuring at time $\theta_i^n$ as a potential service completion by server $k$ and define $\{S_{P,k}^n(t)=\sum_{i=1}^{S_P^n(t)} \delta_k(\kappa_i^n), t\geq 0\}$ as the potential service completion process by server $k$ in the $n$th system. To convert the potential times to actual service completion times, we need to apply another splitting by checking whether server $k$ is busy right before a potential service completion time. Hence, we represent the service completion process for server $k$ as
\[
D_k^n(t)=\int_0^t (1-I_k^n(s-))dS_{P,k}^n(s)=\sum_{i=1}^{S_P^n(t)}(1-I_k^n(\theta_i^n-))\delta_{k}(\kappa_i^n), t\geq 0.
\]
Similarly, taking $\{N^n(t), t\geq 0\}$ to be Poisson processes with rate 1, independent of other processes defined so far, we can represent the abandonment process for the $n$th system as 
 \[
  \left\{N^n\left(\gamma\int_0^t(X^n(s)-n)^+ds\right), t\geq 0\right\}.
 \]
 
 Now, we can write the dynamics of the system length process $\{X^n(t), t\geq 0\}$ as
 \[
  X^n(t)=X^n(0)+A^n(t)-\sum_{k=1}^nD_k^n(t)-N^n\left(\gamma\int_0^t(X^n(s)-n)^+ds\right), \forall t\geq 0, n\in \Nbb.
 \]
  We define the scaled system length process $\hat{X}^n(t):=\frac{X^n(t)-n}{\sqrt{n}}$ and after some simple algebraic manipulations, we get
  \begin{align}
  \begin{split}
  \hat{X}^n(t) &=\hat{X}^n(0) + \frac{A^n(t)-\lambda_n}{\sqrt{n}}+\frac{\lambda_n t-n\bar{\mu} t}{\sqrt{n}}-  \frac{\sum_{k=1}^n D_k^n(t)-\tilde{\mu}_k^n\int_0^t(1-I_k^n(s))ds}{\sqrt{n}}-\sum_{k=1}^n\frac{\tilde{\mu}_k^n\int_0^tI_k^n(s)ds}{\sqrt{n}}\\
  &\quad+\sum_{k=1}^n\frac{\tilde{\mu}_k^n t-\bar{\mu} t}{\sqrt{n}}-\frac{N^n\left(\gamma\int_0^t(X^n(s)-n)^+ds\right)-\gamma\int_0^t(X^n(s)-n)^+ds}{\sqrt{n}}-\frac{\gamma\int_0^t(X^n(s)-n)^+ds}{\sqrt{n}}.
  \end{split}
  \label{Eq:scaled_system}
  \end{align}
 We have the following assumption on  the initial conditions $X^n(0)$:
 \begin{asm}
 \label{asm:initial}
 $\hat{X}^n(0)=\displaystyle \frac{X^n(0)-n}{\sqrt{n}}\Rightarrow \xi_0$.
 \end{asm}
We also define the filtrations $\mathbf{F}^n=\{\CalF_t^n: t\geq 0\}$ as
\[
\CalF_t^n:=\sigma\left(\{\tilde{\mu}_k^n\}_{k=1}^n, X^n(0), \{U_i^n\}_{i=1}^{S_P^n(t)}, A^n(s), S_P^n(s), \{I_k^n(s)\}_{k=1}^n, N^n\left(\gamma\int_0^s(X^n(u)-n)^+du\right):0\leq s\leq t\right) 
\]
for all $t\geq 0$. We are now ready to further analyze the idleness process and define the fairness process. 
  
 \subsection{Analysis of the Total Idleness Process}  
  
 In this section, we analyze the scaled versions of the total idleness process which we define as
 \[
 \bar{I}^n(t):=\frac{\sum_{k=1}^n I_k^n(t)}{n}, \hat{I}_k^n(t):=\frac{I_k^n(t)}{\sqrt{n}}, \text{ for all } 1\leq k\leq n \text{ and }\hat{I}^n(t):=\sum_{k=1}^n \hat{I}_k^n(t), t\geq 0.
 \]
To be able to manipulate the system equations and define the fairness process properly, we need to understand how the total idleness behaves. We prove that the total idleness experienced by time $t$, scales with $\sqrt{n}$ as $n$ increases by showing the tightness of the appropriately scaled maximum number of idle servers.
 
 \begin{lemma}\label{lem:tightness_idle}
  For any fixed $T>0$ and nonidling policy, $\left\{|\hat{I}^n|_{*,T}\right\}_{n\in\Nbb}$ is tight. 
  \end{lemma}   
  \begin{proof}
 We need to prove that for any $\epsilon>0$, there exists a $K_\epsilon>0$, such that 
 \[
  \Pbb(|\hat{I}^n|_{*,T}>K_\epsilon)<\epsilon, \text{ for all }n\in\Nbb.
 \]
 Using \eqref{Eq:nonidling_basic}, 
   \begin{align*}
 \sum_{k=1}^nI^n_k(t) &=\left(X^n(0)-n+A^n(t)- \sum_{k=1}^nD_k^n(t)-N^n\left(\gamma\int_0^t(X(s)-n)^+ds\right)\right)^-\\
 &\leq \left(X^n(0)-n+A^n(t)- \sum_{k=1}^n S_{P,k}^n\left(\tilde{\mu}_k t\right)-N^n\left(\gamma\int_0^t(X(s)-n)^+ds\right)\right)^-,
   \end{align*}
  which implies
   \begin{align}
   \begin{split}
   |\hat{I}^n(t)|_{*,T}& \leq \left|\frac{X^n(0)-n}{\sqrt{n}}\right| +\left|\frac{A^n(t)-\lambda_nt}{\sqrt{n}}\right|_{*,T}+\left|\frac{ (\sum_{k=1}^nS_{P,k}^n(t)-\sum_{k=1}^n\tilde{\mu}_kt)}{\sqrt{n}}\right|_{*,T}
   +\left|\frac{\sum_{k=1}^n\tilde{\mu}_kt-\lambda_nt}{\sqrt{n}}\right|_{*,T}\\&\quad+\left|\frac{N^n(\gamma\int_0^t(X^n(s)-n)^+ds)-\gamma\int_0^t(X^n(s)-n)^+ds}{\sqrt{n}}\right|_{*,T} + \left|\frac{\gamma\int_0^t(X^n(s)-n)^+ds}{\sqrt{n}}\right|_{*,T}.\label{eq:Itight}
   \end{split}
   \end{align}
   Assumptions~\ref{asm:HW} and~\ref{asm:initial}, and the functional central limit theorem for renewal processes guarantees that the first four terms on the righthand side of~\eqref{eq:Itight} are tight. The fifth term corresponds to the supremum of a martingale with predictable quadratic variation 
   \[
  \left\langle\frac{N^n(\gamma\int_0^t(X^n(s)-n)^+ds)-\gamma\int_0^t(X^n(s)-n)^+ds}{\sqrt{n}}\right\rangle=\frac{\gamma\int_0^t(X^n(s)-n)^+ds}{n}.
   \]
 Hence, using Lenglart-Rebolledo inequality (c.f.  Lemma 5.7 in \cite{ptw07}), the tightness of the last two terms in~\eqref{eq:Itight} are spontaneously proved once we show the tightness of 
 \[
  \left|\frac{\int_0^t(X^n(s)-n)^+ds}{\sqrt{n}}\right|_{*,T}.
 \]
 To do so, we define $L^n:=\{0\leq s\leq T:X^n(s)-n>0\}$. The paths of $X^n(\cdot)$ are in $\Dbb_{\Rbb}[0,\infty)$ and hence, we can express $L^n$ as the union of disjoint intervals as $L^n=\bigcup_{i=1}^\infty[\alpha_i^n, \beta_i^n)$.
   \[
   \frac{\int_0^T(X^n(s)-n)^+ds}{\sqrt{n}}=\frac{\int_{L^n}(X^n(s)-n)ds}{\sqrt{n}}=\frac{\sum_{i=1}^\infty\int_{\alpha_i^n}^{\beta_i^n}(X^n(s)-n)ds}{\sqrt{n}}.
   \]
   Note that each $\alpha_i^n$ is the start of a busy period. Hence, if $s\in[\alpha_i^n, \beta_i^n)$
   \[
   X^n(s)\leq X^n(\alpha_i^n)+A^n(s)-A(\alpha_i^n)- S_P^n(s)+S_P^n(\alpha_i^n).
  \]
   Hence, 
   \begin{align*}
   \frac{\int_0^T(X(s)-n)^+ds}{\sqrt{n}}&\leq \sum_{i=1}^\infty\frac{X(\alpha_i^n)-n}{\sqrt{n}}+\sum_{i=1}^\infty\left(\frac{\int_{\alpha_i^n}^{\beta_i^n}\left(A(s)-A(\alpha_i^n)-\lambda^n(s-\alpha_i^n)\right)ds}{\sqrt{n}}\right)\\
   &\quad- \sum_{i=1}^\infty\int_{\alpha_i^n}^{\beta_i^n}\frac{\left(S_P^n(s)-S_P^n(\alpha_i^n)-\sum_{k=1}^n\tilde{\mu}_k^n(s-\alpha_i^n)\right)+(\lambda^n-\sum_{k=1}^n\tilde{\mu}_k^n)(s-\alpha_i^n)}{\sqrt{n}}ds\\
   &\leq T\frac{\max\{X^n(0)-n, 1\}}{\sqrt{n}}+2T\sup_{0\leq s\leq T}\left\{\frac{|A^n(s)-\lambda^ns|}{\sqrt{n}}\right\}\\&\quad+2T\sup_{0\leq s\leq T}\left\{\frac{|S_P^n(s)-\sum_{k=1}^n\tilde{\mu}_k^ns|}{\sqrt{n}}\right\}+\frac{T|\lambda^n-\sum_{k=1}^n\tilde{\mu}^n_k|}{\sqrt{n}}.
   \end{align*}
   Using the functional central limit theorem for Poisson and renewal processes and continuous mapping theorem along with the fact that supremum is continuous in $J_1$ topology, the right-hand side converges weakly and hence is tight. This proves $\left\{|\hat{I}^n|_{*,T}\right\}$ is tight.
  \end{proof}
  \begin{cor}\label{cor:fluidtight}
   $|\bar{I}^n|_{*,T}\stackrel{p}{\rightarrow} 0$ and $\displaystyle\left|\frac{\int_0^t(X^n(s)-n)^+ds}{n}\right|_{*,T}\stackrel{p}{\rightarrow} 0$.
   \end{cor}
   \begin{proof}
   For any $\epsilon>0$, $\Pbb(|\bar{I}^n|_{*,T}>\epsilon)=\Pbb(|\hat{I}^n|_{*,T}>\sqrt{n}\epsilon)\to 0$. The proof of the second claim is similar. 
   \end{proof}
   The proof of the following lemma is similar in nature to that of Lemma 3.1 in \cite{atar08}. 
   \begin{lemma}\label{lem:cumuleffort} For any nonidling policy we have
   \[
   \sup_{0\leq t\leq T}\frac{\sum_{k=1}^n\mu_k\int_0^t I_k(s)ds}{n}\stackrel{p}{\rightarrow} 0.
   \]
   for any given $T$. 
   \end{lemma}
   \begin{proof}
   We know that $I_k(t)=\sqrt{I_k(t)}$. Using Cauchy-Schwarz inequality 
   \begin{align*}
   \sup_{0\leq t\leq T}\frac{\sum_{k=1}^n\mu_k\int_0^t I_k(s)ds}{n}&=\frac{\sum_{k=1}^n\mu_k\int_0^T I_k(s)ds}{n}\\
   &=\int_0^T \left(\sum_{k=1}^n \frac{\mu_k \sqrt{I_k(s)}}{n}\right)ds\\
   &\leq \int_0^T \left(\sqrt{\sum_{k=1}^n \frac{\mu^2_k}{n}}\sqrt{\sum_{k=1}^n\frac{I_k(s)}{n}}\right)ds\\
   &=\left(\sqrt{\sum_{k=1}^n \frac{\mu^2_k}{n}}\right) \int_0^T \sqrt{\frac{I(s)}{n}}ds\\
   &\leq \left(\sqrt{\sum_{k=1}^n \frac{\mu^2_k}{n}}\right) T\sqrt{\frac{|I|_{*,T}}{n}}
   \end{align*}
   
   Hence, our result follows from Lemma~\ref{lem:tightness_idle}. 
   \end{proof}
   
 Our key observation is that how the total idleness is distributed among different servers plays the key role in the analysis of many server queueing systems with random service rates. In the next section, we introduce the fairness process which keeps track of the distribution of total cumulative idleness among servers. 
   
   \subsection{The Fairness Process}
   
   In this section, we concentrate on the analysis of 
   \[
   \sum_{k=1}^n\frac{\tilde{\mu}_k^n\int_0^tI_k^n(s)ds}{\sqrt{n}},
   \]
   the fifth term on the righthand side of~\eqref{Eq:scaled_system}.  For any $\epsilon\geq 0$ and $n\in\Nbb$ we define the random times
   \[
    \tau_\epsilon^n := \inf\left\{t: \int_0^t \hat{I}^n(s)ds>\epsilon\right\}.
   \]
   Now, we are ready to define the measure-valued fairness process. Intuitively, the fairness process keeps track of how the total cumulative idleness is distributed among different servers in the system. Specifically, if the system has been idle for some time by $t$ (i.e., $t>\tau^n_0$), the fairness process, $\eta^n_t(A)$, represents the proportion of cumulative idleness experienced by servers whose service rates, $\tilde{\mu}_k^n$s, are in $A\in \CalB(\Rbb_+)$. When the cumulative experienced idleness is 0, this proportion is not well-defined, and we choose a probability measure $\zeta\in \calP$ as a placeholder and set $\eta^n_t(A)=\zeta(A)$ for any $t\leq\tau_0$ and $A\in \CalB(\Rbb_+)$. We also define the $\epsilon$-shifted version of the fairness process as the process which is equal to $\zeta$ when $t\leq \tau_\epsilon$ and is equal to $\eta^n_t$ for $t>\tau_\epsilon$. Formally, we define the fairness process for the $n$th system and its $\epsilon$-shifted version as
    \[
    \eta^{n}_t(A):=\left\{ \begin{array}{ll}
    \displaystyle \frac{\int_0^t\sum_{k=1}^n\delta_{\mu_k}(A)I_k(s)ds}{\int_0^t\sum_{k=1}^n I_k(s)ds} &\mbox{if }t>\tau_0^n\\
    \zeta(A) &\mbox{if }t\leq \tau_0^n
    \end{array}\right.,\]
    \begin{equation}
    \mathcal{S}^n_\epsilon\eta^{n}_t(A):=\left\{ \begin{array}{ll}
    \displaystyle \eta^n_t(A) &\mbox{if }t>\tau_\epsilon^n\\
    \zeta(A) &\mbox{if }t\leq \tau_\epsilon^n
    \end{array}\right..
    \label{eq:shift}
    \end{equation}
    The fairness process and its shifted versions assume values in the space of probability measures and have paths in $G_\calP[0,\infty)$. We are interested in the convergence of the fairness process in some sense. Unfortunately, the sequence of fairness measures is not tight in the Skorokhod-$J_1$ topology (or in any of the four topologies introduced in \cite{sko56}). Hence, we resort to the following idea of convergence relying on the shifted processes.
    \begin{dfn}\label{dfn:limitingprocess}
     Assume that $\tau_\epsilon^n\Rightarrow \tau_\epsilon$ for any $\epsilon\geq 0$ and define $\mathcal{S}_\epsilon\eta$  by replacing $\tau^n_\epsilon$ with $\tau_\epsilon$ in~\eqref{eq:shift}. We say that $\{\eta_t\}_{t\in \Rbb_{+}}$ is the limiting fairness process if for any $\epsilon>0$
    \[
    \mathcal{S}^n_\epsilon\eta^{n} \Rightarrow \mathcal{S}_\epsilon\eta,\mbox{on $G_\calP[0,\infty)$ endowed with Skorokhod-$J_1$ topology}.
    \]
    
    \end{dfn}
    Lemma~\ref{lem:tightness_idle} and the continuous mapping theorem guarantee that $\{\tau^n_\epsilon\}_{n\in \Nbb}$ is tight. Also, if $\epsilon_1<\epsilon_2$ and $\mathcal{S}^n_{\epsilon_1}\eta^{n} \Rightarrow \mathcal{S}_{\epsilon_1}\eta$, we also have $\mathcal{S}^n_{\epsilon_2}\eta^{n} \Rightarrow \mathcal{S}_{\epsilon_2}\eta$. Hence, rather than considering all real $\epsilon>0$, we can concentrate on any sequence $\{\epsilon_n\}_{n\in \Nbb}$ with $\epsilon_n\to 0$. For the sake of simplicity, we take $\epsilon_n=1/n$ for the rest of this paper. 
    
    Now, our next step is to prove the tightness of the $\epsilon$-shifted processes for any $\epsilon>0$.  Jakubowski~\cite{jak86} provides useful criteria to prove the tightness of measure-valued processes, by converting the problem to that of real-valued processes. This result has been used by Kaspi and Ramanan~\cite{kasram11} in the study of measure-valued processes arising in many server queueing systems and our approach is similar. For completeness, we state Jakubowski's criteria below.
    
   \begin{thm}[Jakubowski~\cite{jak86}]
   A sequence of stochastic processes $\{\eta^{n}\}_{n\in\Nbb}$ taking values in $D_\calP[0.T]$ is tight if and only if:
   \begin{description}
   \item[J1. (Compact Containment Condition)] For each $\rho, T>0$, there exists a compact set $\mathcal{K}_\rho\subset \calP$ such that
   \[
 	\liminf_{n\to \infty}\Pbb(\eta^{n}_t\in \mathcal{K}_\rho, \text{ for all }t\in[0,T])>1-\rho.  
   \]
   \item[J2.] There exists a family of functions $\mathbb{F}$ such that 
   \begin{enumerate}[i.]
   \item $H\in \mathbb{F}: \calP\to \Rbb$ and $\mathbb{F}$ separates points in $\calP$,
   \item $\mathbb{F}$ is closed under addition,
   \item For any fixed $H\in\mathbb{F}$, the sequence of functions $\{h^n(t):=H(\eta^n_t), \mbox{ for all }t \in \Rbb\}_{n\in \Nbb}$ is tight in $G_{\Rbb}[0,\infty)]$ endowed with usual Skorokhod-$J_1$ topology. 
   
   \end{enumerate}
   \end{description}
   \end{thm}
 
Following the example of Kaspi and Ramanan~\cite{kasram11}, we choose the family of functions 
$$\mathbb{F}:=\{H: \exists f\in C_{\Rbb}^b[0,\infty) \mbox{ such that }H(\eta)=\llangle f, \eta\rrangle \mbox{ for all }\eta\in \calP\}.$$
We impose a slightly stronger uniform integrability condition on our fairness processes.
 \begin{asm}\label{asm:uniformboundedness}
 For any given $\rho, \varrho>0$ and $T>0$, there exists an $M$ such that 
 \[
 \liminf_{n\to \infty}\Pbb\left(\int_M^\infty \mu d\eta_t^n(\mu)<\varrho, \text{ for all }t\in[0,T]\right)>1-\rho.  
 \]
 \end{asm}
 This condition is slightly stronger than J1 and holds for any policy if the service rates have bounded support. Now, we are ready to prove the tightness of shifted fairness processes. 
 \begin{thm}\label{thm:shiftedtightness}
 For any $\epsilon>0$, the sequence of $\epsilon$-shifted processes $\{\mathcal{S}_\epsilon^n\eta^n\}_{n\in\Nbb}$ is tight.
 \end{thm} 
 \begin{proof}
 Again it is enough to show the tightness restricting our attention to all finite intervals $[0,T]$. Assumption~\ref{asm:uniformboundedness} guarantees that the condition J1 holds. As $\eta$ takes values in the space of probability measures, $\{\llangle f, \mathcal{S}_\epsilon^n\eta^{n}\rrangle\}_{n\in\Nbb}$ is also bounded for any $f\in C_b[0,\infty)$. To prove J2, we need to show the modulus of continuity condition holds, i.e.,  for any given $\rho, \varsigma>0$, there exists a $C_{\varrho}$ such that for any $\varrho<C_\varrho$
 \[
 \liminf_{n\to \infty} \Pbb(w'(\llangle f, \mathcal{S}_\epsilon^n\eta^{n}\rrangle,\varrho)\geq \rho)<\varsigma,
 \]
 where $w'(\llangle f, \mathcal{S}_\epsilon^n\eta^{n}\rrangle,\varrho):=\inf w(\llangle f, \mathcal{S}_\epsilon^n\eta^{n}\rrangle,(t_{i},t_{i+1}])$ and $$w(\llangle f, \mathcal{S}_\epsilon^n\eta^{n}\rrangle,(t_{i},t_{i+1}]):=\sup_{t_{i}<s,t\leq t_{i+1}}|\llangle f, \mathcal{S}_\epsilon^n\eta^{n}\rrangle_t-\llangle f, \eta^{n}\rrangle_s|.$$  The infimum is taken over all partitions $\{t_i\}_{i\in \Nbb}$ of $[0,T]$ with $|t_{i+1}-t_{i}|\geq\varrho$ for all $i$. For more details on modulus of continuity condition, we refer the reader to Chapter 3 in~\cite{bil99}, noting that the condition we provide is modified for left-continuous functions. 
 
 For any given $t_2>t_1>\tau^n_{\epsilon}$, we have
 \begin{align*}
 &|\llangle f, \mathcal{S}_\epsilon^n\eta^{n}_{t_2}\rrangle-\llangle f, \mathcal{S}_\epsilon^n\eta^{n}_{t_1}\rrangle|\\&\quad\quad=\left|\frac{\displaystyle\int_0^{t_2}\sum_{k=1}^nf(\tilde{\mu}_k^n)\hat{I}_k^n(s)ds}{\displaystyle\int_0^{t_2}\sum_{k=1}^n\hat{I}_k^n(s)ds}-\frac{\displaystyle\int_0^{t_1}\sum_{k=1}^nf(\tilde{\mu}_k^n)\hat{I}_k^n(s)ds}{\displaystyle\int_0^{t_1}\sum_{k=1}^n\hat{I}_k^n(s)ds}\right|\\
 &\quad\quad=\left|\frac{\displaystyle\int_0^{t_1}\sum_{k=1}^nf(\tilde{\mu}_k^n)\hat{I}_k^n(s)ds+\int_{t_1}^{t_2}\sum_{k=1}^nf(\tilde{\mu}_k^n)\hat{I}_k^n(s)ds}{\displaystyle\int_0^{t_1}\sum_{k=1}^n\hat{I}_k^n(s)ds+\int_{t_1}^{t_2}\sum_{k=1}^n\hat{I}_k^n(s)ds}-\frac{\displaystyle\int_0^{t_1}\sum_{k=1}^nf(\tilde{\mu}_k^n)\hat{I}_k^n(s)ds}{\displaystyle\int_0^{t_1}\sum_{k=1}^n\hat{I}_k^n(s)ds}\right|\\
 &\quad\quad=\left|\frac{\displaystyle\left(\int_{t_1}^{t_2}\sum_{k=1}^nf(\tilde{\mu}_k^n)\hat{I}_k^n(s)ds\right)\left(\int_0^{t_1}\sum_{k=1}^n\hat{I}_k^n(s)ds\right)-\left(\int_{0}^{t_1}\sum_{k=1}^nf(\tilde{\mu}_k^n)\hat{I}_k^n(s)ds\right)\left(\int_{t_1}^{t_2}\sum_{k=1}^n\hat{I}_k^n(s)ds\right)}{\displaystyle\left(\int_0^{t_2}\sum_{k=1}^n\hat{I}_k^n(s)ds\right)\left(\int_0^{t_1}\sum_{k=1}^n\hat{I}_k^n(s)ds\right)}\right|\\
 \end{align*}
 Hence, for any partition of $[0,T]$ such that $t_1=0, t_2=\tau_\epsilon$ and $|t_{i+1}-t_{i}|=\varrho$ for $i\geq 2$,
 \[
 w(\llangle f, \mathcal{S}_\epsilon^n\eta^{n}\rrangle,(t_{i},t_{i+1}])\leq \frac{2|f|_{*,\infty}|I|_{*,\infty}\varrho}{\epsilon}
 \]
 By definition $\tau_\epsilon>\epsilon/|I|_{*,\infty}$ and
 \[
 \Pbb(w'(\llangle f, \mathcal{S}_\epsilon^n\eta^{n}\rrangle,\varrho)\geq \rho)\leq \Pbb\left(|I|_{*,\infty}>\max\left\{\frac{\rho\epsilon}{2|f|_{*,\infty}\varrho}, \frac{\epsilon}{\varrho}\right\}\right).
 \]
 The tightness of $|I|_{*,\infty}$ guarantees that this probability approaches 0 as $\varrho\to 0$, which proves our theorem. \end{proof}
 
 In our analysis of the system length process, we are mainly concerned with the process $\llangle\iota, \eta^n\rrangle$ which is equal to the expected value of a random variable following the probability distribution $\eta^n_t$ at any time $t$ and our next result shows that tightness of the shifted versions of these processes is implied by Theorem~\ref{thm:shiftedtightness}. 
 \begin{cor}
 For any $\epsilon>0$, the set of real-valued processes $\{\llangle \iota, \mathcal{S}_\epsilon^n\eta^n\rrangle\}_{n\in\Nbb}$ is tight.
 \end{cor}
 \begin{proof}
 The compact containment condition is implied by Assumption~\ref{asm:uniformboundedness} as for any given $\rho, \varrho>0$ and $T>0$
 \[
 \liminf_{n\to \infty}\Pbb\left(\llangle \iota, \mathcal{S}_\epsilon^n\eta^n_t\rrangle<M+\varrho, \text{ for all }t\in[0,T]\right)>1-\rho.
 \]
 We now prove the modulus of continuity condition, we first define the truncated identity function as $\iota_M(\mu)=\mu\delta_{\{\mu\leq M\}}(\mu)+M\delta_{\{\mu<M\}}(\mu)$. Using that for any $t_1,t_2\geq 0$,
 \begin{align*}
 \llangle\iota, \mathcal{S}_\epsilon^n\eta^n_{t_2}\rrangle-\llangle\iota, \mathcal{S}_\epsilon^n\eta^n_{t_1}\rrangle &\leq\llangle\iota_M, \mathcal{S}_\epsilon^n\eta^n_{t_2}\rrangle-\llangle\iota_M, \mathcal{S}_\epsilon^n\eta^n_{t_1}\rrangle+\int_M^\infty\mu d\mathcal{S}_\epsilon^n\eta^n_{t_2}+\int_M^\infty\mu d\mathcal{S}_\epsilon^n\eta^n_{t_1}.
 \end{align*}
 we write
 \[
 \Pbb(w'(\llangle \iota, \mathcal{S}_\epsilon^n\eta^{n}\rrangle,\varrho)\geq \rho)\leq \Pbb(w'(\llangle \iota_M, \mathcal{S}_\epsilon^n\eta^{n}\rrangle,\varrho)\geq \rho-2\varsigma)+\Pbb\left(\int_M^\infty \mu d\eta_t^n(\mu)>\varsigma, \text{ for some }t\in[0,T]\right).
 \]
 For any $\rho>0$ we can find $\Delta, \varrho$ and $M$ such that the righthand side is less than $\varsigma$ for all $n$, and our result is proved. 
 \end{proof}
 
 Choosing $\epsilon_n=1/n$ and using a diagonal argument, Theorem~\ref{thm:shiftedtightness} implies that any sequence of $\{\eta^n\}_{n\in \Nbb}$ has a subsequence where shifted fairness processes weakly converges for all $\epsilon>0$. However, doing the analysis for each $\epsilon>0$ can be tedious for specific applications. We define the measure-valued scaled total cumulative idleness process as
 \[
 \hat{C}^n(t,A):= \sum_{k=1}^n \delta_{\tilde{\mu}_k^n}(A)\int_{0}^{t}\hat{I}_k^n(s)ds, \mbox{for all }t\geq 0, A\in \CalB(\Rbb_{+}), n\in \Nbb.
 \] 
We now show that it is possible to identify the limiting fairness process, by examining $\hat{C}^n(t,A)$ rather than the $\epsilon$-shifted versions of the fairness process. 
 
 \begin{lemma}\label{lem:equivconv}
 Suppose for all $A\in \mathcal{B}(\Rbb_+)$ and $T>0$, 
 \begin{equation}
 \lim_{n\to \infty} \sup_{0\leq t\leq T}\left|\hat{C}^n(t,A)-\hat{C}(t, A)\right|\toP 0. 
 \label{Eq:ToPCondition}
 \end{equation}
 Then 
 \[
 \eta_t(A)=\left\{\begin{array}{ll}
 \displaystyle\frac{\hat{C}(t, A)}{\hat{C}(t,\Rbb_+)} &\mbox{if }\hat{C}(t,\Rbb_{+})>0\\
 \zeta(A) &\mbox{otherwise}
 \end{array}\right.
 \]
 is the limiting fairness process.
 \end{lemma}
 
 \begin{proof}
 It is clear that $\{\eta(t,\cdot)\}$ is a probability measure valued process. For any Borel set $A\in \mathcal{B}(\Rbb_{+})$ processes $\{\hat{C}^n(t,A), t\geq 0\}$ are continuous processes and converge in probability to  $\{\hat{C}(t, A), t\geq 0\}$ in $\Cbb_{\Rbb_+}[0,\infty)$. The stopping times $\tau_\epsilon^n=\inf\{t>0:  \hat{C}^n(t,\Rbb_+)>\epsilon\}$ and can be thought of as a continuous function of $\hat{C}^n(\cdot,\Rbb_+)$. Using the continuous mapping theorem,  $$\tau_\epsilon^n\toP \tau_\epsilon=\inf\{t>0: \hat{C}(t,\Rbb_+)>0\}.$$ 
 
For any fixed $\epsilon>0$, define 
 \[
 \Lambda^n(t)=\left\{\begin{array}{ll}\displaystyle
 \frac{\tau_\epsilon}{\tau_\epsilon^n} &\mbox{if }0\leq t\leq \tau_\epsilon^n\\
 \displaystyle \frac{T-\tau_\epsilon}{T-\tau_\epsilon^n}t+\frac{T(\tau_\epsilon-\tau_\epsilon^n)}{T-\tau_\epsilon^n} &\mbox{if }\tau_\epsilon^n<t\leq T.
 \end{array}
 \right.
 \]
Clearly, $|\Lambda^n(t)-t|\toP 0$. We now prove for any $A\in \CalB(\Rbb_{+})$
 \begin{equation}
 \lim_{n\to\infty}\left|\eta^n(t,A)-\eta(\Lambda^n(t), A)\right|\toP 0, \mbox{in $G_{\Rbb}[0,T]$.}
 \label{Eq: ConvPA}
 \end{equation}
 For $t\leq \tau^n_\epsilon$, we have $\left|\eta^n(t,A)-\eta(\Lambda^n(t), A)\right|=0$ for all $A$. $\Lambda^n(t)$ is defined so that $\tau_\epsilon>\tau_\epsilon^n$ if and only if $t<\Lambda^n(t)$ for all $t$.  For $t\geq \tau^n_\epsilon$
 \begin{align*}
 \left|\eta^n(t,A)-\eta(\Lambda^n(t), A)\right|&=\delta(\tau_\epsilon>\tau_\epsilon^n)\left|\eta^n(t,A)-\eta(\Lambda^n(t), A)\right|+\delta(\tau_\epsilon\leq\tau_\epsilon^n)\left|\eta^n(t,A)-\eta^n(\Lambda(t), A)\right|\\
 &\leq \delta(\tau_\epsilon\leq \tau_\epsilon^n)\left(\left|\eta^n(t,A)-\eta(t, A)\right|+\left|\eta(\Lambda^n(t),A)-\eta(t, A)\right|\right)\\
 &\quad +\delta(\tau_\epsilon>\tau_\epsilon^n)\left(\left|\eta^n(\Lambda^n(t),A)-\eta(\Lambda^n(t), A)\right|+\left|\eta^n(\Lambda^n(t),A)-\eta^n(t, A)\right|\right)\\
 &=\delta(\tau_\epsilon\leq \tau_\epsilon^n)\left(\left|\frac{\hat{C}^n(t,A)}{\hat{C}^n(t,\Rbb_{+})}-\frac{\hat{C}(t,A)}{\hat{C}(t,\Rbb_{+})}\right|+\left|\frac{\hat{C}^n(\Lambda^n(t),A)}{\hat{C}^n(\Lambda^n(t),\Rbb_{+})}-\frac{\hat{C}^n(t,A)}{\hat{C}^n(t,\Rbb_{+})}\right|\right)\\
 &\quad+\delta(\tau_\epsilon>\tau_\epsilon^n)\left(\left|\frac{\hat{C}^n(\Lambda^n(t),A)}{\hat{C}^n(\Lambda^n(t),\Rbb_{+})}-\frac{\hat{C}(\Lambda^n(t),A)}{\hat{C}(\Lambda^n(t),\Rbb_{+})}\right|+\left|\frac{\hat{C}^n(\Lambda^n(t),A)}{\hat{C}^n(\Lambda^n(t),\Rbb_{+})}-\frac{\hat{C}^n(t,A)}{\hat{C}^n(t,\Rbb_{+})}\right|\right).
 \end{align*}
We think of the righthand side of the last equality as four summands and if we can show converge of each summand to 0 in probability, then we can conclude that for all $A$, $\eta^n(\cdot, A)\toP\eta(\cdot, A)$ in $G_{\Rbb_+}[0,T]$. 
 For the first summand, we have
 \begin{align*}
 \left|\frac{\hat{C}^n(t,A)}{\hat{C}^n(t,\Rbb_{+})}-\frac{\hat{C}(t,A)}{\hat{C}(t,\Rbb_{+})}\right|&=\left|\frac{\hat{C}(t,\Rbb_{+})(\hat{C}^n(t,A)-\hat{C}(t,A))+\hat{C}(t,A)(\hat{C}^n(t,\Rbb_{+})-\hat{C}(t,\Rbb_{+}))}{\hat{C}(t,\Rbb_{+})\hat{C}^n(t,\Rbb_{+})}\right|\\
 &\leq\frac{\hat{I}^n(t)T\left(\left|\hat{C}^n(t,A)-\hat{C}(t,A)\right|+\left|\hat{C}^n(t,\Rbb_{+})-\hat{C}(t,\Rbb_{+})\right|\right)}{\epsilon^2},
 \end{align*}
 and for the second summand, 
 \begin{align*}
 \left|\frac{\hat{C}^n(\Lambda^n(t),A)}{\hat{C}^n(\Lambda^n(t),\Rbb_{+})}-\frac{\hat{C}^n(t,A)}{\hat{C}^n(t,\Rbb_{+})}\right|&= \left|\frac{\hat{C}^n(\Lambda^n(t),A)\hat{C}^n(t,\Rbb_{+})-\hat{C}^n(t,A)\hat{C}^n(\Lambda^n(t),\Rbb_{+})}{\hat{C}^n(\Lambda^n(t),\Rbb_{+})\hat{C}^n(t,\Rbb_{+})}\right|\\
 &= \left|\frac{\hat{C}^n(\Lambda^n(t),A)\hat{C}^n(t,\Rbb_{+})-\hat{C}^n(t,A)\hat{C}^n(\Lambda^n(t),\Rbb_{+})}{\hat{C}^n(\Lambda^n(t),\Rbb_{+})\hat{C}^n(t,\Rbb_{+})}\right|\\
 &\leq \frac{|\hat{C}^n(t,\Rbb_{+})-\hat{C}^n(\Lambda^n(t),\Rbb_{+})|}{\epsilon}+\frac{|\hat{C}^n(t,A)-\hat{C}^n(\Lambda^n(t),A)|}{\epsilon}\\
 &\leq \frac{2\hat{I}^n(t)|t-\Lambda^n(t)|}{\epsilon},
 \end{align*} 
which implies convergence of both terms to 0 in probability. The convergence of third and fourth terms can be proven similarly, which in turn imply \eqref{Eq: ConvPA}.

Using the convergence for individual $A\in\CalB(\Rbb_{+})$, we now prove the convergence of the measure-valued processes. To prove our result, we use part 2 of Theorem 5.3 in Mitoma~\cite{mitoma83}, which translated into our case, states that if
 \begin{enumerate}
 \item For each $f\in\Cbb_b(\Rbb_{+})$, $\{\llangle f, \mathcal{S}_\epsilon^n\eta^{n}\rrangle_t, t\geq 0\}$ is tight in $\Dbb_{\Rbb}[0.T]$.
 \item For each $f_1, f_2, \ldots, f_m\in \Cbb_b(\Rbb_{+})$ and $t_1, t_2, \ldots, t_m\in [0,T]$, 
 \[
 (\llangle f_1, \mathcal{S}_\epsilon^n\eta^{n}_{t_1}\rrangle,\llangle f_2, \mathcal{S}_\epsilon^n\eta^{n}_{t_2}\rrangle,\cdots,\llangle f_m, \mathcal{S}_\epsilon^n\eta^{n}_{t_m}\rrangle)\Rightarrow (\llangle f_1, \mathcal{S}_\epsilon\eta_{t_1}\rrangle,\llangle f_2, \mathcal{S}_\epsilon\eta_{t_2}\rrangle,\cdots,\llangle f_m, \mathcal{S}_\epsilon\eta_{t_m}\rrangle) \in \Rbb^m
 \]
 \end{enumerate}
 then $ \mathcal{S}_\epsilon^n\eta^{n}\Rightarrow  \mathcal{S}_\epsilon\eta$. 
 
 We proved the tightness condition in Theorem~\ref{thm:shiftedtightness}. We prove the second part by proving convergence in probability which is slightly stronger than what is deserved. For any $\varrho>0$
 \begin{align*}
 &\lim_{n\to\infty}\Pbb(\left|(\llangle f_1, \mathcal{S}_\epsilon^n\eta^{n}_{t_1}\rrangle,\cdots,\llangle f_m, \mathcal{S}_\epsilon^n\eta^{n}_{t_m}\rrangle)-(\llangle f_1, \mathcal{S}_\epsilon\eta_{t_1}\rrangle,\cdots,\llangle f_m, \mathcal{S}_\epsilon\eta_{t_m}\rrangle)\right|>\varrho)\\ &\qquad\qquad\qquad\qquad\qquad\qquad\qquad\qquad\qquad\leq \lim_{n\to\infty}\sum_{i=1}^m \Pbb(\left|\llangle f_i, \mathcal{S}_\epsilon^n\eta^{n}_{t_i}\rrangle-\llangle f_i, \mathcal{S}_\epsilon\eta_{t_i}\rrangle\right|>\frac{\varrho}{m})\\
 \end{align*}
 The functions $f_i(x)\in\Cbb_{b}$, i.e., $|f_i(x)|<b_i$ for some $b_i$ and $f_i(x)$ can be approximated by a simple function such that 
 \[
 \left|f_i(x)-\sum_{j=1}^{k_i} a_j^i \delta_x(A_j^i)\right|<\frac{\epsilon}{6m}, \mbox{ for }1\leq i\leq m.
 \]
 Therefore, 
 \begin{align*}
 \Pbb(\left|\llangle f_i, \mathcal{S}_\epsilon^n\eta^{n}_{t_i}\rrangle-\llangle f_i, \mathcal{S}_\epsilon\eta_{t_i}\rrangle\right|>\frac{\varrho}{m})&\leq \Pbb\left(\left|\llangle f_i, \mathcal{S}_\epsilon^n\eta^{n}_{t_i}\rrangle-\llangle f_i, \mathcal{S}_\epsilon\eta_{\Lambda^n(t_i)}\rrangle\right|\right.\\&\quad\quad\quad+
 \left.\left|\llangle f_i, \mathcal{S}_\epsilon\eta_{t_i}\rrangle-\llangle f_i, \mathcal{S}_\epsilon\eta_{\Lambda^n(t_i)}\rrangle\right|>\frac{\varrho}{m}\right)\\
 &\leq \Pbb(\left|\llangle f_i, \mathcal{S}_\epsilon^n\eta^{n}_{t_i}\rrangle-\llangle f_i, \mathcal{S}_\epsilon\eta_{\Lambda^n(t_i)}\rrangle\right|>\frac{\varrho}{2m})+\Pbb(2b_i|t_i-\Lambda^n(t_i)|>\frac{\varrho}{m})\\
 &\leq \Pbb\left(\left|\llangle f_i, \mathcal{S}_\epsilon^n\eta^{n}_{t_i}\rrangle-\llangle \sum_{j=1}^{k_i} a_j^i \delta_x(A_j^i), \mathcal{S}_\epsilon^n\eta^n_{t_i}\rrangle\right|>\frac{\varrho}{6m}\right)\\
 &\quad +\Pbb\left(\left|\llangle f_i, \mathcal{S}_\epsilon\eta_{\Lambda^n(t_i)}\rrangle-\llangle \sum_{j=1}^{k_i} a_j^i \delta_x(A_j^i), \mathcal{S}_\epsilon\eta_{\Lambda^n(t_i)}\rrangle\right|>\frac{\varrho}{6m}\right)\\
 &\quad +\Pbb\left(\left|\llangle \sum_{j=1}^{k_i} a_j^i \delta_x(A_j^i), \mathcal{S}_\epsilon^n\eta^n_{t_i}\rrangle-\llangle \sum_{j=1}^{k_i} a_j^i \delta_x(A_j^i), \mathcal{S}_\epsilon\eta_{\Lambda^n(t_i)}\rrangle\right|>\frac{\varrho}{6m}\right)\\
 &\quad+\Pbb(2b_i|t_i-\Lambda^n(t_i)|>\frac{\varrho}{m})\\
 &\leq \sum_{j=1}^{k_i} \Pbb\left(\left|\eta^n(t_i,A_j^i)-\eta(\Lambda^n(t_i), A_j^i)\right|>\frac{\varrho}{6b_imk_i}\right)\\
 &\quad+\Pbb(2b_i|t_i-\Lambda^n(t_i)|>\frac{\varrho}{m}), 
 \end{align*}
 which converges to 0 as implied by~\eqref{Eq: ConvPA}.
 \end{proof}
 
  \subsubsection*{Limiting Fairness Measures for Fastest-Server First and Slowest-Server-First Policies}
  
  We use Lemma~\ref{lem:equivconv} to derive the limiting fairness processes for fastest-server-first and slowest-server-first policies, where the system controller has the knowledge of individual service rates and routes the arriving customer to the fastest and slowest idle servers, respectively.  
  
  \begin{thm}\label{thm:fsf_fairness}
  If $\zeta=\delta_{\mu_\min}$. where $\mu_\min := \essinf(\tilde{\mu}_k^n)$, then the limiting fairness measure for Fastest-Server-First policy is
  \[
  \eta_t(A)=\delta_{\mu_\min}(A) \mbox{ for all }t\geq 0.
  \]
  
  \end{thm}
  \begin{proof}
  We prove that for any $\epsilon>0$ $A^\epsilon=[\mu_\min+\epsilon, \infty)$ and $T>0$
  \[
  \left|\hat{C}^n(t, A^\epsilon)\right|_{*, T}=\hat{C}^n(T,A^\epsilon)\stackrel{p}{\rightarrow} 0.
  \]
 We define $L^n$ as in Lemma~\ref{lem:tightness_idle} and use the fact that it can be partitioned into disjoint intervals as
  \[
 \hat{C}^n(T,A^\epsilon)=\sum_{i=0}^\infty\int_{\alpha^n_i}^{\beta^n_i}\frac{\sum_{k=1}^n\delta_{\mu_k}(A^\epsilon)I^n_k(s)ds}{\sqrt{n}} \leq \sum_{i=1}^\infty(\beta^n_i-\alpha^n_i)|\hat{I}^n|_{*,t}.
  \]
  Since, we know that $|\hat{I}^n|_{*,t}$ is tight, if we can show that $\sum_{i=1}^\infty(\beta^n_i-\alpha^n_i)\stackrel{p}{\to} 0$, our result will follow. The FSF policy guarantees that if there is an arrival when at least one of the servers with service rate in $A^\epsilon$ is idle, then the new arrival is routed to one of the servers in this set. Hence, within the interval $(\alpha_i^n,\beta_i^n)$ with $i=1, 2,\ldots$, the number of service completions by the servers with service rate in $A^\epsilon$ should be greater than the number of arrivals. Hence, 
  \begin{align*}
  \Pbb(\sum_{i=1}^\infty|\beta^n_i-\alpha^n_i|>0)&\leq \Pbb\left(\sup_{0<s_1<s_2<t}\sum_{k=1}^n \delta_{\mu_k}(A^\epsilon) \left(\frac{S_{P.k}^n(s_2)-S_{P,k}^n(s_1)}{\sqrt{n}}\right)-\frac{A(s_2)-A(s_1)}{\sqrt{n}}>0\right)\\
  &=\Pbb\left(\sup_{0<s_1<s_2<t}\sum_{k=1}^n \delta_{\mu_k}(A^\epsilon) \frac{S_{P.k}^n(s_2)-S_{P,k}^n(s_1)-\tilde{\mu}_k(s_2-s_1)}{\sqrt{n}}\right.\\
  &\qquad\qquad\qquad+\frac{(\sum_{k=1}^n\tilde{\mu}_k-n\bar{\mu})(s_2-s_1)}{\sqrt{n}}-\frac{\sum_{k=1}^n(1-\delta_{\tilde{\mu}_k}(A^\epsilon))\tilde{\mu}_k(s_2-s_1)}{\sqrt{n}}\\
  &\left.\qquad\qquad\qquad-\frac{A(s_2)-A(s_1)-\lambda_n(s_2-s_1)}{\sqrt{n}}+\frac{(n\bar{\mu}-\lambda_n)(s_2-s_1)}{\sqrt{n}}>0\right). 
  \end{align*}
  
  The term $\frac{\sum_{k=1}^n(1-\delta_{\tilde{\mu}_k}(A^\epsilon))\tilde{\mu}_k(s_2-s_1)}{\sqrt{n}}$ diverges to infinity as $n\to \infty$. Using the functional central limit theorems for Poisson and renewal processes and continuous mapping theorem for the supremum, we know that the other terms converge in distribution and hence tight. This implies that the probability approaches 0 as $n\to \infty$.
  \end{proof}
  
  As will be shown in Theorem~\ref{thm:conv_system} in the next section, the diffusion approximation provided in Theorem 2.2 in Atar~\cite{atar08} also implies the same limiting fairness measure for the fastest-server-first policy. Our result is  slightly stronger as we do not require the support of the service rates to be bounded. However, for Assumption~\ref{asm:uniformboundedness} to be satisfied for the slowest-server-first policy, we need to impose the boundedness condition. The rest of the proof of the following theorem is similar to the proof of Theorem~\ref{thm:fsf_fairness} and we omit it here. 
  
  \begin{thm}\label{thm:ssf_fairness}
    Suppose that $\zeta=\delta_{\mu_\max}$,  where $\mu_\max := \esssup(\tilde{\mu}_k^n)<\infty$. Then, the limiting fairness measure for Slowest-Server-First policy is
    \[
    \eta_t(A)=\delta_{\mu_\max}(A) \mbox{ for all }t\geq 0.
    \]
   
    \end{thm}
 
\begin{remark}\label{rem:sko}
Unfortunately, the sequences of fairness processes do not converge  in any of the four topologies provided by Skorokhod~\cite{sko56} in general. We try to provide some insight into this claim by considering fastest-server-first policy as an example.  Assume $\hat{X}^n(0)$ and their limit $\xi_0$ is bounded away from 0, hence $\tau^n_0>0$. We take $\zeta=\delta_{\mu_\min}$ and as we prove in Theorem~\ref{thm:fsf_fairness} the limiting fairness process is continuous and is also equal to $\delta_{\mu_\min}$. Choosing $A=[\mu_\min+\epsilon, \infty)$ such that $\Pbb(\tilde{\mu}_k^n\in A)>0$, 
\begin{align*}
&\Pbb\left(\sup_{0<t<T}|\eta^n(t)-\eta(t)|>\epsilon\right)\geq \Pbb\left(\exists t>0\mbox{ such that } \eta^n(t)=\delta_{\mu^*} \mbox{ where }\mu^*\in A| \CalF_0\right)=\frac{\sum_{k=1}^n \delta_{\tilde{\mu}_k^n}(A)\tilde{\mu}_k^n}{\sum_{k=1}^n \tilde{\mu}_k^n}.
\end{align*}
The equality follows using the basic properties of minimum of exponential random variables and realizing that a time $t$ as stated above exists if the first server to become idle is in set $A$.
Hence, it is not possible to find any $\Lambda(\cdot)$ to show convergence of fairness processes in $J_1$ and $J_2$ topologies. The distance between graphs of sample paths also is greater than $\epsilon$ in the case above, making convergence in $M_1$ and $M_2$ topologies impossible. 
\end{remark}  
  
 \subsection{Convergence of the System Length Process}
 
 In this section, we analyze the weak limit of the scaled system length processes $\hat{X}^n(t)$. Theorem~\ref{thm:shiftedtightness} and the completeness of the space of probability measures guarantee that for any sequence of fairness processes, there exists a subsequence with a limiting fairness process. The uniqueness of the limiting process can be proven if the finite dimensional distributions of the process converge. Henceforth, we assume that the sequence of fairness processes has a unique limiting process in the sense in Definition~\ref{dfn:limitingprocess}.
 
 A general approach in proving limit theorems for queueing systems is to use the structure of Equation~\eqref{Eq:scaled_system} in conjuction with the continuous mapping theorem  (c.f. Theorem 2.7 in Billingsley~\cite{bil99} and Section 4.1 in~\cite{ptw07}). The nonstandard definition of the limiting fairness processes in Definition~\ref{dfn:limitingprocess} makes it harder for us to use this theorem directly. The proof of the continuous mapping theorem relies on the Skorokhod Representation Theorem for which we provide an extension in the following lemma. 
 \begin{lemma}\label{lem:repext}
 Let $S_k$ be Polish spaces and $\Xi^n=(X_1^n, X_2^n, \ldots)$ be random elements where $X_k^n\in S_k$ for all $k, n\in \Nbb$ and $\{X_{k,n}\}_{n\in\Nbb}$ is tight for each $k$. There exist $n_1<n_2<\cdots$ in $\Nbb$ and random variables $Y_{nj}$ defined on the same probability space $(\hat{\Omega},\hat{\CalF}, \hat{\Pbb})$, taking values in $S_{n_j}$ such that for each $j$ the family $(Y_{kj})_{k\in \Nbb}$ has the same joint law with $(X_{kn_j})_{k\in \Nbb}$ and for each $k\in\Nbb$, $Y_{kj}$ converges almost surely as $j\to\infty$.
 \end{lemma} 
 \begin{proof}
 Let $\tilde{S}=S_1\times S_2 \times \cdots$, an infinite Cartesian product, which is a Polish space. Then $\tilde{X}_n=(X_{1n}, X_{2n}, \cdots)$ are $\tilde{S}$-valued random elements for each $n\in \Nbb$. As $\{X_k^n\}_{n\in \Nbb}$ is tight, there exists a compact set $E_k\subset S_k$ such that $\Pbb(X_{kn}\notin E_k \text{ for some }n)<2^{-k}\epsilon$ for each $k\in \Nbb$. This implies that $\Pbb(X_{kn}\notin E_{k}\text{ for some } k,n\in\Nbb)<\epsilon$. Using Tychonoff's theorem (see Chapter 5, Munkres~\cite{mun00}), $E_1\times E_2\times \cdots$ is compact in $\tilde{S}$. Hence, we conclude that $\{\tilde{X}_n\}_{n\in \Nbb}$ is also tight. Hence, using Skorokhod's representation theorem, we can find $n_1<n_2<\cdots$ and $\tilde{Y}_j$ on a probability space $(\hat{\Omega},\hat{\CalF}, \hat{\Pbb})$ which has the same distribution as $\tilde{X}_{n_j}$ and converges almost surely as $j\to \infty$. Hence, the lemma follows. 
 \end{proof}

 Now, we are ready to prove our main result.
 \begin{thm}\label{thm:conv_system}
 Suppose that $\{\eta_{t}\}_{t\in \Rbb}$ is the limiting fairness process for the sequence of queueing systems as defined in Definition~\ref{dfn:limitingprocess}. Then, the scaled process $\hat{X}^n$ converges weakly to $\xi$, where $\xi$ is the solution of 
 \begin{equation}
 \xi(t)=\xi_0+(\bar{\mu}\sqrt{c_a^2+1}) W(t)+\beta t +\llangle \iota, \eta_t \rrangle\int_0^t\xi^-(s)ds-\gamma \int_0^t\xi ^+(s)ds, \text{ for all } t\geq 0,
 \end{equation} 
 where $c_a^2:=\text{var}(u_k^n)/\bar{\mu}^2$, i.e., the coefficient of variation of $u_k^n$s, $\beta := \hat{\lambda} - \tilde{\sigma}$, $\tilde{\sigma}$ is distributed normally with mean 0 and variance $\text{var}(\tilde{\mu})$.
 \end{thm}
 \begin{proof}
 We focus on Equation~\eqref{Eq:scaled_system} and use the martingale method outlined in~\cite{ptw07}. The processes 
 \begin{align*}
 \hat{M}_1^n(t)&:=\frac{\sum_{k=1}^n D_k^n(t)-\tilde{\mu}_k^n\int_0^t(1-I_k^n(s))ds}{\sqrt{n}}\\
 \hat{M}_2^n(t)&:=\frac{N^n\left(\gamma\int_0^t(X^n(s)-n)^+ds\right)-\gamma\int_0^t(X^n(s)-n)^+ds}{\sqrt{n}},
 \end{align*}
  corresponding to the fourth and sixth terms in~\eqref{Eq:scaled_system} are martingales with predictable quadratic variations
  \begin{align*}
 \langle \hat{M}_1^n\rangle_t&=\frac{\sum_{k=1}^n\tilde{\mu}_k^n\int_0^t(1-I_k^n(s))ds}{n}\\
 \langle \hat{M}_2^n\rangle_t&=\frac{\gamma\int_0^t(X^n(s)-n)^+ds}{n},
  \end{align*}
  respectively. 
  Using Corollary~\ref{cor:fluidtight},
 \[
\frac{\sum_{k=1}^n\tilde{\mu}_k t - \sum_{k=1}^n\tilde{\mu}_k\int_{0}^t I^n_k(s)ds}{n}\stackrel{p}{\rightarrow }\bar{\mu} t \text{ in }D_{\Rbb_+}[0,\infty) \mbox{ and }\left|\frac{\int_0^t(X^n(s)-n)^+ds}{n}\right|_{*,T}\stackrel{p}{\rightarrow} 0.
 \]
 Hence, using the martingale central limit theorem
 \[
 \hat{M}_1^n(t)\Rightarrow \bar{\mu} W_2(t) \text{ and } \hat{M}_2^n(t)\stackrel{p}{\to} 0.
 \]
 And as a result of central limit theorem for renewal processes and standard central limit theorem, we have
 \[
 \hat{A}^n(t)=\frac{A^n(t)-\lambda_nt}{\sqrt{n}}\Rightarrow c_a \bar{\mu}W_1(t) \text{ and }\hat{\sigma}^n=\frac{\tilde{\mu}_k^n -\bar{\mu}}{\sqrt{n}}\Rightarrow \tilde{\sigma}.
 \]
 Taking $\epsilon_k=1/k$ for all $k\in \Nbb$ and $\Xi^n=(\hat{X}^n(0)\hat{A}^n, \hat{M}_1^n, \hat{M}_2^n, \hat{\sigma}^n, \mathcal{S}^n_{\epsilon_1}\eta^{n}, \mathcal{S}^n_{\epsilon_2}\eta^{n}, \ldots)$, Lemma~\ref{lem:repext} implies that we can find $\breve{\Xi}^k=(\breve{X}^n(0),\breve{A}^k, \breve{M}_1^k, \breve{M}_2^k, \breve{\sigma}^k, \mathcal{S}^k_{\epsilon_1}\breve{\eta}^{k}, \mathcal{S}^k_{\epsilon_2}\breve{\eta}^{k}, \ldots)$ such that $\breve{\Xi}^k$ has the same distribution as $\Xi^{n_k}$ and 
 \[
 \breve{\Xi}^k\to (\breve{\xi}_0, c_a\bar{\mu}\breve{W}_1, \bar{\mu}\breve{W}_2, 0, \breve{\sigma}, \mathcal{S}_{\epsilon_1}\breve{\eta}, \mathcal{S}_{\epsilon_2}\breve{\eta}, \ldots) ,
 \]
 where $(\breve{\xi}_0, c_a\bar{\mu}\breve{W}_1, \bar{\mu}\breve{W}_2, 0, \breve{\sigma}, \mathcal{S}_{\epsilon_1}\breve{\eta}, \mathcal{S}_{\epsilon_2}\breve{\eta}, \ldots)$ has the same joint distribution as $$(\xi_0, c_a\bar{\mu}W_1, \bar{\mu}W_2, 0, \tilde{\sigma}, \mathcal{S}_{\epsilon_1}\eta, \mathcal{S}_{\epsilon_2}\eta, \ldots).$$ As the convergence occurs in Skorokhod-$J_1$ topology, Theorem 4.1 in~\cite{whi80} implies that for any fixed $\epsilon>0$ we can find a common sequence of homeomorphisms $\Lambda^n:[0,T]\to[0,T]$ such that with probability 1,
 \[
 \left|\breve{A}^n(t)-\bar{\mu}c_a\breve{W}_1(\Lambda^n(t))\right|_{*,T}\vee\left|\breve{M}_1^n(t)-\bar{\mu}\breve{W}_2(\Lambda^n(t))\right|_{*,T}\vee \left|\llangle\iota, \mathcal{S}_\epsilon\breve{\eta}^n_t\rrangle-\llangle\iota, \mathcal{S}_\epsilon\breve{\eta}_{\Lambda^n(t)}\rrangle\right|_{*,T} \vee\left|\dot{\Lambda}^n(t)-1\right|_{*,T}\to 0,
 \]
 as $n\to \infty$, where $\dot{\Lambda}^n$ denotes the derivative of $\Lambda$ with respect to $t$. Also, Assumption~\ref{asm:uniformboundedness} and the tightness proved in Lemma~\ref{lem:tightness_idle} imply that for any $\rho>0$, there exists an $K_\rho>0$ such that 
 \[
 \Pbb\left(\sup_{n\in \Nbb}\left\{\left|\llangle\iota, \breve{\eta}^n_t\rrangle\right|_{*,T}\vee \left|\breve{X}^n(t)\right|_{*,T}\right\}>K_\rho\right)<\rho.
 \] Let $\breve{\xi}(t)$ be the unique strong solution of
 \[
 \breve{\xi}(t)=\breve{\xi}_0+c_a\bar{\mu}\breve{W}_1(t) + \bar{\mu}\breve{W}_2(t) + \beta t +\llangle \iota, \breve{\eta}_t \rrangle\int_0^t\breve{\xi}^-(s)ds-\gamma \int_0^t\breve{\xi}^+(s)ds, \text{ for all } t\geq 0.
 \]
 Defining $\breve{X}$ by replacing $\Xi$ with $\breve{\Xi}$ in~\eqref{Eq:scaled_system}, we need to prove that for any $\rho, \varrho>0$, we can find a $N_{\rho, \varrho}$ such that 
 $n>N_{\rho, \varrho}$ implies
 \[
 \Pbb(d_S(\breve{X}^n, \breve{\xi})>\varrho)<\rho.
 \]
 Choosing $K_{\rho/2}$ as defined above 
 \begin{equation}
 \label{eq:convprob}
 \Pbb\left(d_S(\breve{X}^n, \breve{\xi})>\varrho\right)\leq \Pbb\left(d_S(\breve{X}^n, \breve{\xi})>\varrho, \sup_{n\in \Nbb}\left\{\left|\llangle\iota, \breve{\eta}^n_t\rrangle\right|_{*,T}\vee \left|\breve{X}^n(t)\right|_{*,T}\right\}\leq K_{\rho/2}\right)+\frac{\rho}{2}.
 \end{equation}
 Hence, we concentrate on the scenarios corresponding to the first term on the righthand side of~\eqref{eq:convprob} and assume 
 \[
 \sup_{n\in \Nbb}\left\{\left|\llangle\iota, \breve{\eta}^n_t\rrangle\right|_{*,T}\vee \left|\breve{X}^n(t)\right|_{*,T}\right\}\leq K_{\rho/2}
 \]
 For any $\varepsilon>0$, there is a sufficiently large $N_0$, such that for any $n>N_0$
 \begin{align*}
 |\breve{X}^n(t)-\breve{\xi}(\Lambda^n(t))|&\leq \varepsilon +\gamma\left|\int_0^t(\breve{X}^n(s))^+ds-\int_0^{\Lambda^n(t)}\breve{\xi}(s)^+ds\right|\\&\quad+ \left|\llangle\iota, \breve{\eta}^n_t\rrangle\int_0^t(\breve{X}^n(s))^-ds-\llangle\iota, \breve{\eta}_{\Lambda^n(t)}\rrangle\int_0^{\Lambda(t)}\breve{\xi}(s)^-ds\right|\\
 &\leq \varepsilon +\gamma\left|\int_0^t\left((\breve{X}^n(s))^+-\breve{\xi}\left(\Lambda^n(s)\right)^+\right)ds+\int_0^{t}(1-\dot{\Lambda}^n(s))\breve{\xi}\left(\Lambda^n(s)\right)^+ds\right|\\
 &\quad+ \left|\left(\llangle\iota, \breve{\eta}^n_t\rrangle-\llangle\iota, \mathcal{S}_\epsilon\breve{\eta}^n_t\rrangle\right)\int_0^t(\breve{X}^n(s))^-ds\rrangle_t\right.\\
 &\quad\quad\quad\quad\quad\quad\quad\quad\quad\quad\quad\quad\left.-\left(\llangle\iota, \breve{\eta}_{\Lambda^n(t)}\rrangle-\llangle\iota, \mathcal{S}_\epsilon\breve{\eta}_{\Lambda^n(t)}\rrangle\right)\int_0^{\Lambda(t)}\breve{\xi}(s)^-ds\right|\\
 &\quad+ \left|\left(\llangle\iota, \mathcal{S}_\epsilon\breve{\eta}^n_t\rrangle-\llangle\iota, \mathcal{S}_\epsilon\breve{\eta}_{\Lambda^n(t)}\rrangle\right)\int_0^t(\breve{X}^n(s))^-ds\right.\\
  &\quad\quad\quad\quad\quad\quad\quad\quad\quad\quad\quad\quad\left.-\llangle\iota, \mathcal{S}_\epsilon\breve{\eta}_{\Lambda^n(t)}\rrangle\int_0^{t}\left((\breve{X}^n(s))^--\breve{\xi}(\Lambda(s))^-\right)ds\right|\\
 &\quad+ \left|\llangle\iota, \mathcal{S}_\epsilon\breve{\eta}_{\Lambda^n(t)}\rrangle\int_0^{t}\dot{\Lambda}(s)\breve{\xi}(\Lambda(s))^-ds\right|\\
 &\leq \varepsilon +\gamma\int_0^t\left|(\breve{X}^n(s))-\breve{\xi}\left(\Lambda^n(s)\right)\right|ds+\varepsilon K_{\rho/2}T+4\epsilon K_{\rho/2}+ \varepsilon K_{\rho/2}\\&\quad+K_{\rho/2}\int_0^{t}\left|\breve{X}^n(s)-\breve{\xi}(\Lambda(s))\right|ds+ \varepsilon\left(K_{\rho/2}\right)^2T
 \end{align*}
 Now, using Gronwall's inequality (c.f. Lemma 4.1 in~\cite{ptw07}), we have
 \[
 d_S(\breve{X}^n, \breve{\xi})\leq \sup_{0\leq t\leq T}|\breve{X}^n(t)-\breve{\xi}(\Lambda^n(t))|\leq (\varepsilon(1+K_{\rho/2}(1+T+K_{\rho/2}T)+4\epsilon K_{\rho/2})e^{(\gamma + K_{\rho/2})T}
 \]
 Hence, choosing  $\varepsilon$ and $\epsilon$ appropriately and $N_0$ large enough, we prove our result. 
 \end{proof}

\section{Totally Blind Routing Policies}
\subsection{A Related Martingale}
In this section, our goal is to identify fairness processes corresponding to some specific policies. The key term in defining the fairness process is the cumulative idle time by time $t$ experienced by servers whose service rate fall in set $A$. To analyze the total idleness process further, we use the doubly stochastic process $\{S_P^n(t)\}_{t\geq 0}$ that represents the potential service completion times and define
\[
\phi_i^n=\inf\{t-\theta_i^n: I_{\kappa_i^n}(t)=0, t>\theta_i^n\} \text{ for all $i\in \Nbb$ and $n\in\Nbb$}.
\]
The random variable $\phi_i^n$ represents the length of the idle period experienced by the server who becomes idle at time $\theta_i^n$. If an actual service completion occurs at $\theta_i^n$, then $I_{\kappa_i^n}(\theta_i^n-)=1$ and $\phi_i^n>0$, and if $\theta_i^n$ is not an actual service completion time, i.e., $I_{\kappa_i^n}(\theta_i^n-)=0$, then $\phi_i^n=0$. Similarly, we define 
\[
\phi_{-k}^n=\inf\{t>0: I_k^n(t)=0\}, \text{for all $1\leq k \leq n$ and $n\in\Nbb$},
\]
which represents the duration of the idle period experienced by servers who are idle at time 0. Hence, defining 
\[
\Upsilon_i^n(A)=\left\{\begin{array}{ll}
1    &\text{if }\tilde{\mu}_{\kappa_i^n}\in A\\
0    &\text{otherwise},
\end{array}\right.
\]
 we can equivalently represent the total cumulative idleness experienced by servers in $A$ by time $t$ as 
\[
\sum_{k=1}^n\delta_{\tilde{\mu}_k^n}(A)\int_{0}^t I_k^n(s)ds = \sum_{k=1}^n\delta_{\tilde{\mu}_k^n}(A)(\phi_{-k}^n\wedge t)+\sum_{i=1}^{S_P^n(t)}\Upsilon_i^n(A)(\phi_i^n\wedge(t-\theta_i^n)).
\]
To analyze the righthand side further, we introduce the following martingale:

\begin{lemma}\label{lem:martingale} For any Borel set $A\subset \Rbb_+$, 
\begin{align}
\nonumber M_A^n(t)&=n^{-1/2}\left(\sum_{i=1}^{S_P^n(t)}\Upsilon_i^n(A)(\phi_i^n\wedge (t-\theta_i^n))-\sum_{i=1}^{S_P^n(t)}\Ebb\left[\Upsilon_i^n(A)\phi_i^n|\CalF_{\theta_i^n-}^n\right]\right.\\&\qquad\qquad\quad+\left.\sum_{i=1}^{S_P^n(t)}\left(\Upsilon_i^n(A)\Ebb[(\phi_i^n-t+\theta_i^n)^+|\CalF_t^n]\right)\right)
\label{Eq:Martingale}
\end{align} is an $\mathbf{F}^n$-martingale. 
\end{lemma}
\begin{proof}
Being conditional expectations, second and third sums on the right-hand side of \eqref{Eq:Martingale} are clearly $\CalF_t$-measurable. The random variables $(\phi_i^n+\theta_i^n)$ are stopping times for any $0<i\leq S_P(t)$, which implies that the first sum is also measurable with respect to $\CalF_t$. For any $n\geq 1$ and $t>\theta_i$, we have
\[
0\leq X^n(t)<n \text{ and }\phi_i^n-\theta<\inf\{s>t: X^n(s)=n\}.
\]
Considering the discrete-time Markov chain  $Y_k^n=X^n(\sum_{i=1}^ku_i^n/\lambda^n)$ and the positive recurrence of this chain when restricted to be between 0 and $n$ proves that 
$\Ebb[|M_t(A)^n|]<\infty$ for all $t\geq 0$. The rest of the proof relies on the basic relationship that $\phi_i^n\wedge(t-\theta_i^n)=\phi_i^n-(\phi_i^n+\theta_i^n-t)^+$. We have
\begin{align*}
\Ebb[M_A^n(t+s)-M_A^n(t)|\CalF_t]&=\Ebb\left[\sum_{i=1}^{S_P^n(t+s)}\Upsilon_i^n(A)\phi_i^n - \sum_{i=1}^{S_P^n(t+s)}\Upsilon_i^n(A) (\phi_i^n-t-s+\theta_i^n)^+\right.\\&\quad\quad\quad-\sum_{i=1}^{S_P^n(t+s)}\Ebb[\Upsilon_i^n(A)\phi_i^n|\CalF_{\theta_i^n-}] + \sum_{i=1}^{S_P^n(t+s)}\Ebb[\Upsilon_i^n(A)(\phi_i^n-t-s+\theta_i^n)^+|\CalF_{t+s}]\\ &\quad\quad\quad-\sum_{i=1}^{S_P^n(t)}\Upsilon_i^n(A)\phi_i^n + \sum_{i=1}^{S_P^n(t)}\Upsilon_i^n(A) (\phi_i^n-t+\theta_i^n)^+\\
&\quad\quad\quad+\sum_{i=1}^{S_P^n(t)}\Ebb[\Upsilon_i^n(A)\phi_i^n|\CalF_{\theta_i^n-}]-\left.\left.\sum_{i=1}^{S_P^n(t)}\Ebb[\Upsilon_i^n(A)(\phi_i^n-t+\theta_i^n)^+|\CalF_{t}]\right|\CalF_t\right]\\
&=\Ebb\left[ \sum_{i=S_P^n(t)+1}^{S_P^n(t+s)}\Upsilon_i^n(A)\phi_i^n - \sum_{i=1}^{S_P^n(t)}\Upsilon_i^n(A) (\phi_i^n-t-s+\theta_i^n)^+\right.\\
&\quad\quad\quad -\sum_{i=S_P^n(t)+1}^{S_P^n(t+s)}\Upsilon_i^n(A) (\phi_i^n-t-s+\theta_i^n)^+-\sum_{i=S_P^n(t)+1}^{S_P^n(t+s)}\Ebb[\Upsilon_i^n(A)\phi_i^n|\CalF_{\theta_i^n-}] \\
&\quad\quad\quad + \sum_{i=1}^{S_P^n(t)}\Ebb[\Upsilon_i^n(A)(\phi_i^n-t-s+\theta_i^n)^+|\CalF_{t+s}]\\&\quad\quad\quad+\sum_{i=S_P^n(t)+1}^{S_P^n(t+s)}\Ebb[\Upsilon_i^n(A)(\phi_i^n-t-s+\theta_i^n)^+|\CalF_{t+s}]
+ \sum_{i=1}^{S_P^n(t)}\Upsilon_i^n(A) (\phi_i^n-t+\theta_i^n)^+\\&\quad\quad\quad-\left.\left.\sum_{i=1}^{N^n(t)}\Ebb[\Upsilon_i^n(A)(\phi_i^n-t+\theta_i^n)^+|\CalF_{t}]\right|\CalF_t\right]=0
\end{align*}
This proves the lemma. 
\end{proof}

The $\hat{C}^n(t,A)$ is a continuous function of time $t$ for any $A\in \CalB(\Rbb_{+})$ and hence a previsible process. Hence, in Lemma~\ref{lem:martingale}, rather than finding a compensator for the cumulative total idleness process, we find a process who is compensator is the cumulative total idleness process and hence there is no uniqueness claim.

\begin{lemma}\label{lem:limrem}
If for any $T>0$, there exist random variables $\vartheta_{n,T}$ such that $\Ebb[\vartheta_{T}^n]\to 0$ as $n\to \infty$ and 
 \begin{equation}
\sup_{0\leq t \leq T}\sup_{-n\leq i\leq S_P^n(t)}\Ebb\left[(\phi_i^n-t+\theta_i^n)^+|\CalF_{t}^n\right]<\vartheta_{T}^n,
\label{Eq:remainder_cond}
\end{equation}
then 
\[
\sup_{0\leq t\leq S_P^n(T)}\left|\hat{C}^n(t,A)-\frac{\sum_{i=1}^{S_P^n(t)}\Ebb[\Upsilon_i^n(A)\phi_i^n|\CalF_{\theta_i^n-}]}{\sqrt{n}}\right|\toP 0.
\]
\end{lemma}

\begin{proof}
To prove our result we show that the martingale $M_A^n$ and the third term on the righthand side of~\eqref{Eq:Martingale} converge to 0 in probability under the stated conditions. Condition~\eqref{Eq:remainder_cond} imply that for all $T\geq 0$ and $n\in \Nbb$, $\vartheta_{T}^n$ is positive and converges 0 in probability using Markov inequality.  
If the expected maximum jump and the optional quadratic variation of $M_A^n$ converges 0, the latter being in probability, the martingale can be shown to converges 0 in probability using the martingale central limit theorem (c.f. Theorem 8.1 in \cite{ptw07}). We first note that \eqref{Eq:remainder_cond} also implies 
\[
\sup_{1\leq i\leq S_P^n(T)} \Ebb\left[\phi_i^n|\CalF_{\theta_i^n-}^n\right]\toP 0, 
\]
which also imply the expected maximum jump of $M_A^n$ converges 0. Expanding the optional quadratic variation of $M_A^n$, we get
\begin{align}
\nonumber \left[M_A^n\right]&\leq  \frac{2\left[\sum_{i=1}^{S_P^n(t)}\Upsilon_i^n(A)(\phi_i^n\wedge (t-\theta_i^n))\right]}{n}+\frac{2\left[ \sum_{i=1}^{S_P^n(t)}\Ebb\left[\Upsilon_i^n(A)\phi_i^n|\CalF_{\theta_i^n-}^n\right]\right]}{n}\\&\quad+\frac{2\left[ \sum_{i=1}^{S_P^n(t)}\left(\Upsilon_i^n(A)\Ebb[(\phi_i^n-t+\theta_i^n)^+|\CalF_t^n]\right)\right]}{n}.
\label{Eq:OQV}
\end{align}
The first term on the righthand side of~\eqref{Eq:OQV} is the optional quadratic variation of a continuously increasing process and hence is 0 for all $T\geq 0$ and $n\in \Nbb$. The second term is the optional quadratic variation of an increasing pure jump process and 

\[
\displaystyle\frac{\left[ \displaystyle\sum_{i=1}^{S_P^n(t)}\Ebb\left[\Upsilon_i^n(A)\phi_i^n|\CalF_{\theta_i^n-}^n\right]\right]}{n}=\frac{\displaystyle\sum_{i=1}^{S_P^n(t)}\Ebb[\Upsilon_i^n(A)\phi_i^n|\CalF_{\theta_i^n-}]^2}{n}\leq \left(\sup_{1\leq i\leq S_P^n(t)}\Ebb[\Upsilon_i^n(A)\phi_i^n|\CalF_{\theta_i^n-}]\right)^2\displaystyle \frac{S_P^n(t)}{n}\toP 0.
\]
For analyzing the third term, we define  $\varpi_i^n(t)=\inf\{I_{\kappa_i}(s): \theta_i^n\leq s\leq t\}$, for $i\leq S_P^n(t), n\in \Nbb$, which is 1 if the server completing the service at $\theta_i^n$ stays idle until time $t$ and is 0 otherwise. The function $\varpi_i^n(t)$ enables us to explicitly address only the servers who are idle at time $t$ as for any $\theta_i^n\leq S_P^n(t)$, 
\[
\Ebb[(\phi_i^n-t+\theta_i)^+|\CalF_{t}] =\varpi_i(t)\Ebb[(\phi_i^n-t+\theta_i)^+|\CalF_{t}].
\]
We now analyze the optional quadratic variation of the third term on the righthand side of \eqref{Eq:OQV} directly using the definition (c.f. Theorem 3.3 in \cite{ptw07}). Consider any finite partition of $[0,T]$ such that $0=t_0<t_1<\cdots<t_m=T$ and
\[
\begin{array}{l}
\displaystyle n^{-1}\sum_{l=1}^m\left( \sum_{i=1}^{S_P(t_l)}\Upsilon_i(A)\varpi_i(t_l)\Ebb[(\phi_i^n-t_l+\theta_i)^+|\CalF_{t_l}] - \sum_{i=1}^{S_P(t_{l-1})}\Upsilon_i(A)\varpi_i(t_{l-1})\Ebb[(\phi_i^n-t_{l-1}+\theta_i)^+|\CalF_{t_{l-1}}]\right)^2\\
=n^{-1}\displaystyle\sum_{l=1}^m\left( \sum_{i=S_P(t_{l-1}+1)}^{S_P(t_l)}\Upsilon_i(A)\varpi_i(t_l)\Ebb[(\phi_i^n-t_l+\theta_i)^+|\CalF_{t_l}]\right.\\\quad\quad\quad\quad + \left. \displaystyle\sum_{i=1}^{S_P(t_{l-1})}\Upsilon_i(A)\left(\varpi_i(t_{l})\Ebb[(\phi_i^n-t_{l}+\theta_i)^+|\CalF_{t_{l}}]-\varpi_i(t_{l-1})\Ebb[(\phi_i^n-t_{l-1}+\theta_i)^+|\CalF_{t_{l-1}}]\right)\right)^2\\
\displaystyle\leq 2n^{-1}\sum_{l=1}^m\left( \sum_{i=S_P(t_{l-1}+1)}^{S_P(t_l)}\Upsilon_i(A)\varpi_i(t_l)\Ebb[(\phi_i^n-t_l+\theta_i)^+|\CalF_{t_l}]\right)^2\\
\displaystyle\quad\quad\quad + 2n^{-1}\sum_{l=1}^m \left( \sum_{i=1}^{S_P(t_{l-1})}\Upsilon_i(A)\left(\varpi_i(t_{l})\Ebb[(\phi_i^n-t_{l}+\theta_i)^+|\CalF_{t_{l}}]-\varpi_i(t_{l-1})\Ebb[(\phi_i^n-t_{l-1}+\theta_i)^+|\CalF_{t_{l-1}}]\right)\right)^2
\end{array}\]
To analyze  the second term further, we realize that  $\varpi_i(t_l)\leq\varpi_i(t_{l-1})$ for any $i<S_P(t_{l-1})$ and write 
\begin{align*}
&\sum_{l=1}^m \left( \sum_{i=1}^{S_P(t_{l-1})}\Upsilon_i(A)\left(\varpi_i(t_{l})\Ebb[(\phi_i^n-t_{l}+\theta_i)^+|\CalF_{t_{l}}]-\varpi_i(t_{l-1})\Ebb[(\phi_i^n-t_{l-1}+\theta_i)^+|\CalF_{t_{l-1}}]\right)\right)^2\\
&=\sum_{l=1}^m \left(-\sum_{i=1}^{S_P(t_{l-1})}\Upsilon_i(A)(1-\varpi_i(t_l))\varpi_i(t_{l-1})\Ebb[(\phi_i^n-t_{l-1}+\theta_i)^+|\CalF_{t_{l-1}}]\right.\\
&\quad\quad\quad+\left.\sum_{i=1}^{S_P(t_{l-1})}\Upsilon_i(A)(\varpi_i(t_l))\varpi_i(t_{l-1})\left(\Ebb[(\phi_i^n-t_{l}+\theta_i)^+|\CalF_{t_{l}}]-\Ebb[(\phi_i^n-t_{l-1}+\theta_i)^+|\CalF_{t_{l-1}}]\right)\right)^2\\
&\leq 2\sum_{l=1}^m \left(\sum_{i=1}^{S_P(t_{l-1})}\Upsilon_i(A)(1-\varpi_i(t_l))\varpi_i(t_{l-1})\Ebb[(\phi_i^n-t_{l-1}+\theta_i)^+|\CalF_{t_{l-1}}]\right)^2\\
&\quad\quad\quad+2\sum_{l=1}^m \left(\sum_{i=1}^{S_P(t_{l-1})}\Upsilon_i(A)(\varpi_i(t_l))\varpi_i(t_{l-1})\left(\Ebb[(\phi_i^n-t_{l}+\theta_i)^+|\CalF_{t_{l}}]-\Ebb[(\phi_i^n-t_{l-1}+\theta_i)^+|\CalF_{t_{l-1}}]\right)\right)^2\\
\end{align*}
Again considering the second term on the righthand side, we have
\begin{align*}
&\sum_{l=1}^m \left(\sum_{i=1}^{S_P(t_{l-1})}\Upsilon_i(A)(\varpi_i(t_l))\varpi_i(t_{l-1})\left(\Ebb[(\phi_i^n-t_{l}+\theta_i)^+|\CalF_{t_{l}}]-\Ebb[(\phi_i^n-t_{l-1}+\theta_i)^+|\CalF_{t_{l-1}}]\right)\right)^2\\
&\leq 2\sum_{l=1}^m \left(\sum_{i=1}^{S_P(t_{l-1})}\Upsilon_i(A)(\varpi_i(t_l))\varpi_i(t_{l-1})\left(\Ebb[(\phi_i^n-t_{l}+\theta_i)^+|\CalF_{t_{l}}]-\Ebb[(\phi_i^n-t_{l-1}+\theta_i)^+|\CalF_{t_{l}}]\right)\right)^2\\
&\quad\quad+2\sum_{l=1}^m \left(\sum_{i=1}^{S_P(t_{l-1})}\Upsilon_i(A)(\varpi_i(t_l))\varpi_i(t_{l-1})\left(\Ebb[(\phi_i^n-t_{l-1}+\theta_i)^+|\CalF_{t_{l}}]-\Ebb[(\phi_i^n-t_{l-1}+\theta_i)^+|\CalF_{t_{l-1}}]\right)\right)^2\\
&\leq 2\sum_{l=1}^m \left(\sum_{i=1}^{S_P(t_{l-1})}\Upsilon_i(A)(\varpi_i(t_l))\varpi_i(t_{l-1})\left(t_l-t_{l-1}\right)\right)^2\\
&\quad\quad+2\sum_{l=1}^m \left(\sum_{i=1}^{S_P(t_{l-1})}\Upsilon_i(A)(\varpi_i(t_l))\varpi_i(t_{l-1})\left(\Ebb[(\phi_i^n-t_{l-1}+\theta_i)^+|\CalF_{t_{l}}]-\Ebb[(\phi_i^n-t_{l-1}+\theta_i)^+|\CalF_{t_{l-1}}]\right)\right)^2\\
\end{align*}
Aggregating all our calculations above, we get  
\[
\frac{\left[ \sum_{i=1}^{S_P^n(t)}\left(\Upsilon_i^n(A)\Ebb[(\phi_i^n-t+\theta_i^n)^+|\CalF_t^n]\right)\right]}{n}\leq 2V_1^n(t)+4V_2^n(t)+8(V_3^n(t)+V_4^n(t)),
\]
where
\begin{align*}
V_1^n(t)&=\frac{\displaystyle\lim_{\sup|t_l-t_{l-1}|\to 0}\sum_{l=1}^m\left( \sum_{i=S_P(t_{l-1}+1)}^{S_P^n(t_l)}\Upsilon_i(A)\varpi_i(t_l)\Ebb[(\phi_i^n-t_l+\theta_i)^+|\CalF_{t_l}]\right)^2}{n},\\
V_2^n(t)&=\frac{\displaystyle\lim_{\sup|t_l-t_{l-1}|\to 0}\sum_{l=1}^m \left(\sum_{i=1}^{S_P^n(t_{l-1})}\Upsilon_i(A)(1-\varpi_i(t_l))\varpi_i(t_{l-1})\Ebb[(\phi_i^n-t_{l-1}+\theta_i)^+|\CalF_{t_{l-1}}]\right)^2}{n},\\
V_3^n(t)&=\frac{\displaystyle\lim_{\sup|t_l-t_{l-1}|\to 0} \sum_{l=1}^m \left(\sum_{i=1}^{S_P^n(t_{l-1})}\Upsilon_i(A)\varpi_i(t_l)\varpi_i(t_{l-1})\left(t_l-t_{l-1}\right)\right)^2}{n},\\
V_4^n(t)&=\frac{\displaystyle\lim_{\sup|t_l-t_{l-1}|\to 0} \sum_{l=1}^m \left(\sum_{i=1}^{S_P^n(t_{l-1})}\Upsilon_i(A)\varpi_i(t_l)\varpi_i(t_{l-1})\left(\Ebb[(\phi_i^n-t_{l-1}+\theta_i)^+|\CalF_{t_{l}}]-\Ebb[(\phi_i^n-t_{l-1}+\theta_i)^+|\CalF_{t_{l-1}}]\right)\right)^2}{n}.
\end{align*}
The first term $V_1^n$ is the definition of the quadratic variation of a pure jump process and hence,
\[
V_1^n(t)= \frac{\sum_{i=1}^{S_P(t)}\left(\Upsilon_i(A)\Ebb[\phi_i^n|\CalF_{\theta_i^n}]\right)^2}{n}= \left(\sup_{1\leq i\leq S_P^n(t)}\Ebb[\Upsilon_i^n(A)\phi_i^n|\CalF_{\theta_i^n}]\right)^2\displaystyle \frac{S_P^n(t)}{n}\toP 0.
\]
For $V_2^n(t)$,
\[
V_2^n(t)\leq \left(\vartheta_{T}^n\right)^2\frac{\displaystyle\lim_{\sup|t_l-t_{l-1}|\to 0}\sum_{l=1}^m \left(\sum_{i=1}^{S_P(t_{l-1})}\Upsilon_i(A)(1-\varpi_i(t_l))\varpi_i(t_{l-1})\right)^2}{n}
\]
The fraction above defines the optional  quadratic variation of a counting process, which counts the number of  customers routed to a server by time $t$. This counting process is bounded by the arrival process $A^n(t)$ and hence \eqref{Eq:remainder_cond} implies that $V_2^n(t)\toP 0$ in $\Dbb_{\Rbb}[0,T]$. 

For $V_3^n(t)$, write
\begin{align*}
V_3^n(t)&\leq \frac{\displaystyle\lim_{\sup|t_l-t_{l-1}|\to 0} S_P(t)\sum_{l=1}^m \sum_{i=1}^{S_P(t_{l-1})}\left(\Upsilon_i(A)\varpi_i(t_l)\varpi_i(t_{l-1})\left(t_l-t_{l-1}\right)\right)^2}{n}\\
&\leq \frac{\displaystyle 2S_P(t) \sum_{i=1}^{S_P(t)}\displaystyle\lim_{\sup|t_l-t_{l-1}|\to 0}\sum_{l=1}^m\left(t_l-t_{l-1}\right)^2}{n}=0\\
\end{align*}
For the analysis of $V_4^n(t)$, we use a similar approach but this time making explicit use of the fact that the non-zero terms in the summation at any $t_{l-1}$ is bounded by  the number of idle servers at that time. Also, 
\begin{align*}
V_4^n(t)&\leq \lim_{\sup|t_l-t_{l-1}|\to 0} \sum_{l=1}^m 2I^n(t_l)\sum_{i=1}^{S_P(t_{l-1})}\left(\frac{\displaystyle\Upsilon_i(A)\varpi_i(t_l)\varpi_i(t_{l-1})\Ebb[(\phi_i^n-t_{l-1}+\theta_i)^+|\CalF_{t_{l}}]}{n}\right.\\
&\qquad\qquad\qquad\qquad\qquad\qquad\qquad\qquad\qquad\left.-\frac{\Upsilon_i(A)\varpi_i(t_l)\varpi_i(t_{l-1})\Ebb[(\phi_i^n-t_{l-1}+\theta_i)^+|\CalF_{t_{l-1}}]}{n}\right)^2\\
&\leq \frac{2|I^n|_{t,*}^2 \displaystyle\lim_{\sup|t_l-t_{l-1}|\to 0} \sum_{l=1}^m\left(\Ebb[(\phi_i^n|\CalF_{t_{l}}]-\Ebb[\phi_i^n|\CalF_{t_{l-1}}]\right)^2}{n}\\
\end{align*}
Theorem~\ref{lem:cumuleffort} imply that $|I^n|_{t,*}^2/n$ is tight. The limit and summation is the optional quadratic variation of the square integrable martingale $E[\phi_i^n|\CalF_t]$ and using Burkholder-Davis-Gundy inequality (Theorem iV.48 in~\cite{protter05}) with $p=1$ 
\[
\left[\Ebb[\phi_i^n|\CalF_t^n]\right]\leq c_1\Ebb[\sup_{0\leq s\leq t}\sup_{1\leq i\leq S_P^n(t)}\Ebb[\phi_i^n|\CalF_s^n]]\leq c_1\Ebb[\vartheta_{T}^n]
\]
and hence the optional quadratic variation of $E[\phi_i^n|\CalF_t]$ converges 0 in probability. This proves that $[M_A^n]\toP 0$ in probability and using  the martingale functional central limit theorem we conclude that $M_A^n(t)\toP 0$. Using a similar approach, the term on the righhand side of ~\eqref{Eq:Martingale} 
\[
\sum_{i=1}^{S_P^n(t)}\left(\Upsilon_i^n(A)\Ebb[(\phi_i^n-t+\theta_i^n)^+|\CalF_t^n]\right)\leq |\hat{I}^n|_{T,*}\vartheta_{T}^n\toP 0.
\]
This implies 
\[
n^{-1/2}\left(\sum_{i=1}^{S_P^n(t)}\Upsilon_i^n(A)(\phi_i^n\wedge (t-\theta_i^n))-\sum_{i=1}^{S_P^n(t)}\Ebb\left[\Upsilon_i^n(A)\phi_i^n|\CalF_{\theta_i^n-}^n\right]\right)\toP 0\mbox{ in }\Dbb_{\Rbb}[0,T].
\]
Since the limiting function is constant and hence continuous, we conclude that the same convergence also holds with the supremum norm, which proves the lemma.
\end{proof}

\subsection{The Fairness Measure for Totally Blind Policies}
We are now ready to address specific policies when service rates are random. Atar~\cite{atar08} proves that for the Longest-Idle-Server-First policy, the system length process converges to a diffusion as given in \eqref{dfn:limitingprocess}, where 
\[\llangle \iota, \eta \rrangle_t=\frac{\int_{0}^\infty \mu^2dF(\mu)}{\int_{0}^\infty \mu dF(\mu)}, \text{ for all } t\geq 0.\]
This implies that the limiting fairness process is constant through time and 
\[
\eta_t(A)=\frac{\int_A \mu dF(\mu)}{\int_0^\infty \mu dF(\mu)}, \text{for all }t\geq 0.
\]
In this section, under some mild conditions which readily holds for the LISF policy, we prove that the scaled system length process converges to the same limiting diffusion for a more general class of policies, which we call to be ``totally blind policies''. Ward and Armony~\cite{wararm13} describe a blind policy to be a control policy that can depend on the system state but not the system parameters. According to this definition and its usage in the literature, a blind policy is assumed not to have perfect information of parameters, however it can make decisions that have implicit dependencies on parameters through the system state. For our purposes we need to define the notion of blindness a bit more restrictively , to mean the expected idling time of a server who becomes idle at time $\theta_i^n$ is asymptotically the same with or without the knowledge of the service rate of the server. 
\begin{dfn}
A routing policy is a totally blind policy, if and only if 
\[
\sup_{1\leq i\leq S_P^n(T)} \sqrt{n}\left|\Ebb[\phi_i^n|\CalF_{\theta_i^n-}^n]-\Ebb[\phi_i^n|\CalF_{\theta_i^n}^n]\right|\stackrel{p}{\rightarrow} 0
\]
\end{dfn}
\begin{thm}\label{thm:totallyblind}
Under any totally blind policy, i.e.,
\[
\lim_{n\to \infty} \sup_{1\leq i\leq S_P^n(T)} \sqrt{n}\left|\Ebb[\phi_i^n|\CalF_{\theta_i^n-}^n]-\Ebb[\phi_i^n|\CalF_{\theta_i^n}^n]\right|\stackrel{p}{\rightarrow} 0, \mbox{for all $T>0$},
\]
where condition \eqref{Eq:remainder_cond} holds, then 
\[
\eta_t(A)=\frac{\int_A \mu dF(\mu)}{\int_0^\infty \mu dF(\mu)}, \text{for all }t\geq 0.
\]
\end{thm}
\begin{proof}
First, we analyze 
\begin{align*}
\sum_{i=1}^{S_P^n(t)} \frac{\Ebb[\Upsilon_i(A)\phi_i|\CalF_{\theta_i^n-}]}{\sqrt{n}}&=\sum_{i=1}^{S_P^n(t)} \frac{\Ebb[\Ebb[\Upsilon_i(A)\phi_i|\CalF_{\theta_i^n}]|\CalF_{\theta_i^n-}]}{\sqrt{n}}\\
&=\sum_{i=1}^{S_P^n(t)} \frac{\Ebb[\Upsilon_i(A)\Ebb[\phi_i|\CalF_{\theta_i^n}]|\CalF_{\theta_i^n-}]}{\sqrt{n}}\\
&=\sum_{i=1}^{S_P^n(t)} \frac{\Ebb[\Upsilon_i(A)(\Ebb[\phi_i|\CalF_{\theta_i^n}]-\Ebb[\phi_i|\CalF_{\theta_i^n-}])|\CalF_{\theta_i^n-}]}{\sqrt{n}}+\sum_{i=1}^{S_P^n(t)} \frac{\Ebb[\Upsilon_i(A)\Ebb[\phi_i|\CalF_{\theta_i^n-}]|\CalF_{\theta_i^n-}]}{\sqrt{n}}\\
&=\sum_{i=1}^{S_P^n(t)} \frac{\Ebb[\Upsilon_i(A)(\Ebb[\phi_i|\CalF_{\theta_i^n}]-\Ebb[\phi_i|\CalF_{\theta_i^n-}])|\CalF_{\theta_i^n-}]}{\sqrt{n}}+\sum_{i=1}^{S_P^n(t)} \frac{\Ebb[\Upsilon_i(A)|\CalF_{\theta_i^n-}]\Ebb[\phi_i|\CalF_{\theta_i^n-}]}{\sqrt{n}}\\
\end{align*}
Hence, we have
\[
\sum_{i=1}^{S_P^n(t)} \frac{\Ebb[\Upsilon_i(A)\phi_i|\CalF_{\theta_i^n-}]}{\sqrt{n}}-\frac{\Ebb[\Upsilon_i(A)|\CalF_{\theta_i^n-}]\Ebb[\phi_i|\CalF_{\theta_i^n-}]}{\sqrt{n}}=\sum_{i=1}^{S_P^n(t)} \frac{\Ebb[\Upsilon_i(A)(\Ebb[\phi_i|\CalF_{\theta_i^n}]-\Ebb[\phi_i|\CalF_{\theta_i^n-}])|\CalF_{\theta_i^n-}]}{\sqrt{n}}
\]
For any $\epsilon>0$
\[
\begin{array}{l}
\displaystyle\Pbb\left(\left|\sum_{i=1}^{S_P^n(t)} \frac{\Ebb[\Upsilon_i(A)(\Ebb[\phi_i|\CalF_{\theta_i^n}]-\Ebb[\phi_i|\CalF_{\theta_i^n-}])|\CalF_{\theta_i^n-}]}{\sqrt{n}}\right|>\epsilon_1\right)\\
\quad=\displaystyle\Pbb\left(\left|\sum_{i=1}^{S_P^n(t)} \frac{\Ebb[\Upsilon_i(A)(\Ebb[\phi_i|\CalF_{\theta_i^n}]-\Ebb[\phi_i|\CalF_{\theta_i^n-}])|\CalF_{\theta_i^n-}]}{\sqrt{n}}\right|>\epsilon_1, \frac{\left|S_P^n(T)-n\bar{\mu}T\right|}{n}<\epsilon_2\right)\\
\displaystyle\quad\quad\quad +\Pbb\left(\frac{\left|S_P^n(T)-n\bar{\mu}T\right|}{n}>\epsilon_2\right)\\
\quad\leq\displaystyle\Pbb\left(\sum_{i=1}^{S_P^n(t)} \frac{\Ebb[\left|(\Ebb[\phi_i|\CalF_{\theta_i^n}]-\Ebb[\phi_i|\CalF_{\theta_i^n-}])\right||\CalF_{\theta_i^n-}]}{\sqrt{n}}>\epsilon_1, \frac{\left|S_P^n(T)-n\bar{\mu}T\right|}{n}<\epsilon_2\right)\\
\displaystyle\quad\quad\quad +\Pbb\left(\frac{\left|S_P^n(T)-n\bar{\mu}T\right|}{n}>\epsilon_2\right)\\
\quad\leq\displaystyle\Pbb\left( \frac{\displaystyle(\bar{\mu}T+\epsilon_2)\sqrt{n}\sup_{1\leq i\leq S_P(T)}\Ebb[\left|(\Ebb[\phi_i|\CalF_{\theta_i^n}]-\Ebb[\phi_i|\CalF_{\theta_i^n-}])\right||\CalF_{\theta_i^n-}]}{\sqrt{n}}>\epsilon_1\right)\\
\displaystyle\quad\quad\quad +\Pbb\left(\frac{\left|S_P^n(T)-n\bar{\mu}T\right|}{n}>\epsilon_2\right)\\
\end{array}
\]
This implies 
\[
\displaystyle\sum_{i=1}^{S_P^n(t)} \left(\frac{\Ebb[\Upsilon_i(A)\phi_i|\CalF_{\theta_i^n-}]}{\sqrt{n}}-\frac{\Ebb[\Upsilon_i(A)|\CalF_{\theta_i^n-}]\Ebb[\phi_i|\CalF_{\theta_i^n-}]}{\sqrt{n}}\right)\toP 0 \mbox{ in }\Dbb_{\Rbb}[0,T].
\]
Again, since the limiting function (zero function) is constant and hence continuous, this convergence can be taken in the supremum norm. Combining this with Lemma~\ref{lem:limrem} and realizing that 
\[
\Ebb[\Upsilon_i^n(A)|\CalF_{\theta_i^n-}]=\displaystyle\frac{n^{-1}\sum_{k=1}^n\delta_{\tilde{\mu}_k^n}(A)(1-I_k^n(\theta_i^n-))\tilde{\mu}_k^n}{n^{-1}\sum_{k=1}^n\tilde{\mu}_k^n}\to \frac{\int_A \mu dF(\mu)}{\int_0^\infty \mu dF(\mu)} \mbox{ w.p. 1},
\]
we have 
\[
\sup_{0\leq s\leq T}\left|\hat{C}^n(t,A)-\frac{\int_A \mu dF(\mu)}{\int_0^\infty \mu dF(\mu)}\int_0^t\hat{I}(s)ds\right|\toP 0.
\]
Our theorem is then implied by Lemma~\ref{lem:cumuleffort}.
\end{proof}

Intuitively,  as the number of servers approaches infinity, the number of idle servers are negligible compared to the servers who are busy. It is well known from basic probability that if we have two independent exponential random variables, say $U$ and $V$, with rates $\lambda_1$ and $\lambda_2$, the probability that $U$ is less than $V$ is $\lambda_1/(\lambda_1+\lambda_2)$.  Hence, using this basic property and the fact that the idle servers are negligible, the probability that a server becoming idle at any given time belongs to set $A$ is  asymptotically constant in time and equal to 
\[
\frac{\int_A \mu dF(\mu)}{\int_0^\infty \mu dF(\mu)}.
\]
In other words, totally blind policies asymptotically equalize the time to stay idle for all servers once they become idle. The servers with higher service rates become idle more frequently, proportional to their service rates and hence, their share of total cumulative idle time is also proportional to their service rates.

\subsubsection*{Longest-Idle-Server-First Policy}

We show that Longest-Idle-Server-First policy is totally blind and the condition~\ref{Eq:remainder_cond} holds. We know that the server who becomes idle at time $\theta_i^n$ starts the next service exactly after arrival of the next $\sum_{k=1}^n I_k^n(\theta_i^n)$ customers. Hence, 
\[
\Ebb[\phi_i^n|\CalF_t]\leq n^{-1}\left(\sup_{0\leq s\leq T} \Ebb[u_{A^n(s)+1}^n|\CalF_s]+\Ebb[u_1^n]\sup_{0\leq s \leq T}\sum_{k=1}^n I_k^n(s)\right).
\]
and 
\begin{equation}\label{Eq:LISF_expected}
\Ebb[\phi_i^n|\CalF_{\theta_i^n}]=\Ebb[\phi_i^n|\CalF_{\theta_i^n-}]
\end{equation}
for all $i,n\in\Nbb$, as it only depends on the number of idle servers at $\theta_i^n$, which is $\CalF_{\theta_i^n-}$ measurable and the future behavior of the arrivals which is a renewal process. Hence, we have
\begin{thm}
Longest-Idle-Server-First policy is totally blind and satisfies~\eqref{Eq:remainder_cond} and hence the limiting fairness process is as given in Theorem~\ref{thm:totallyblind}.
\end{thm}
\begin{proof}
For any fixed $T>0$, $\sup_{0\leq t\leq T}\Ebb[u_{A^n(s)+1}^n|\CalF_{t}]$ is bounded and $\left\{|\hat{I}^n|_{*,T}\right\}_{n\in\Nbb}$ is tight, which implies the \eqref{Eq:remainder_cond} holds. Also, for any $i$ and $n$, \eqref{Eq:LISF_expected} implies $\Ebb[\phi_i^n|\CalF_{\theta_i^n}]-\Ebb[\phi_i^n|\CalF_{\theta_i^n-}]=0$ and hence LISF policy is totally blind. 
\end{proof}

\section{Conclusion}

In this work, we analyze many-server queueing systems where the service rates of servers in the system are i.i.d. random variables. These systems have been studied in an earlier paper by Atar~\cite{atar08} for two routing policies, longest-idle-server-first and fastest-server-first policies, in an ad hoc manner. Our main contribution is to develop a general framework to analyze many-server queueing systems with service rate uncertainty using  measure-valued stochastic processes. We introduce the fairness process which assumes values in the space of probability measures denoting the proportion of cumulative idleness shared by servers having different service rates. Unfortunately, it is possible to show that the fairness processes cannot be analyzed in the standard Skorokhod topologies and we introduce a modified notion of convergence for these processes. We also include customer abandonments into our analysis, which we believe to be indispensible for systems with service raet uncertainty.  We also show how the limits obtained in~\cite{atar08} can be obtained using martingale methods. 

Even though, the many-server systems with uncertain service rates can be seen as a replacement for servers belonging to pools where servers are identical, we believe that this work also should be extended to include networks of queues with pools of servers, where within a server pool the service rates follow  the same distribution, but may follow different distributions among pools. The introduced diffusion limits can also be used to analyze staffing of many-server systems with service rate uncertainty, and how the variability of the service rate affect the staffing and routing decisions. We also believe that the similar fairness processes can be defined to analyze service systems with heterogeneous and non-exponential service times and data-driven systems where the service rates are either time dependent or learned through data as time evolves. 

\section*{Acknowledgements}
We would like to thank Istvan Gy\"{o}ngi, Sandy Davie and Gon\c{c}alo Dos Reis for the valuable discussions on this problem. 
\bibliographystyle{abbrv}
\bibliography{bb}

\end{document}